\documentclass[12pt]{amsart}
\usepackage[pdftex,pagebackref,letterpaper=true,colorlinks=true,pdfpagemode=none,urlcolor=blue,linkcolor=blue,citecolor=blue,pdfstartview=FitH]{hyperref}

\usepackage{amsmath,amsfonts}
\usepackage{graphicx}
\usepackage{color}
\usepackage{appendix}

\usepackage{standalone}
\newcommand{\norm}[1]{\left\lVert#1\right\rVert}
\newcommand{\abs}[1]{\left\lvert#1\right\rvert}

\reversemarginpar
\usepackage{amssymb}
\usepackage{cite}
\usepackage{enumerate}
\usepackage{gensymb}
\usepackage{float}
\usepackage{amsthm}
\usepackage{siunitx}
\usepackage[mathscr]{euscript}
\newtheorem{theorem}{Theorem}[section]
\newtheorem{corollary}[theorem]{Corollary}
\newtheorem{lemma}[theorem]{Lemma}
\newtheorem{proposition}[theorem]{Proposition}
\newtheorem{definition}[theorem]{Definition}
\newtheorem{remark}[theorem]{Remark}

\usepackage{tikz}
\usepackage{caption}
\usepackage{subcaption}
\usepackage{pgfplots}
\usetikzlibrary{patterns, positioning, arrows}
\usepackage{physics}
\usetikzlibrary{decorations.pathmorphing,calc}
\usepackage{amsrefs}
\usepackage{tcolorbox}
\usepackage{MnSymbol}
\newcommand{\R}{\mathbb{R}}
\newcommand{\N}{\mathbb{N}}

\newcommand{\p}{\partial}
\newcommand{\eps}{\varepsilon}

\usepackage{subfiles}
\numberwithin{equation}{section}
\usepackage{tcolorbox}
\newenvironment{claim}[1]{\par\noindent\textit{Claim\space#1:}}

\begin{document}
\renewcommand{\div}{\textnormal{div}}
\newcommand{\dist}{\textnormal{dist}}
\newcommand{\loc}{\textnormal{loc}}
\newcommand{\osc}{\textnormal{osc}\hspace{0.5mm}}
\newcommand{\diam}{\textnormal{diam}}
\newcommand{\naught}{_{0}}
\newcommand{\supp}{\textnormal{supp}}

\title[Global solutions in the plane]{Classification of global solutions to the obstacle problem in the plane}
\author{Anthony Salib}
\address{Department of Mathematics, University of Duisburg-Essen, Thea-Leymann-Strasse 9, 45127 Essen,Germany.}
\email{anthony.salib@stud.uni-due.de}
\author{Georg Weiss}
\address{Department of Mathematics, University of Duisburg-Essen, Thea-Leymann-Strasse 9, 45127 Essen,Germany.}
\email{georg.weiss@uni-due.de}
\subjclass[2020]{35B08,35R35}
\begin{abstract}
	Global solutions to the obstacle problem were first completely classified in two dimensions by Sakai using complex analysis techniques. Although the complex analysis approach produced a very succinct proof in two dimensions, it left the higher dimensional cases, and even closely related problems in two dimensions, unresolved. A complete classification in dimensions $n\geq 3$ was recently given by Eberle, Figalli and Weiss, forty years after Sakai published his proof. In this paper we give a proof of Sakai's classification result for unbounded coincidence sets in the spirit of the recent proof by Eberle, Figalli and Weiss. Our approach, in particular, avoids the need for complex analysis techniques and offers new perspectives on two-dimensional problems that complex analysis cannot address.
\end{abstract}
\keywords{Obstacle problem, classification of global solutions}
\maketitle
\tableofcontents

\section{Introduction}

The classification of global solutions to the obstacle problem
\begin{equation}\label{obstacleproblem}
	\Delta u = \chi_{\{u>0\}},\hspace{2mm} u \geq 0, \text{ in } \R^n,
\end{equation}
was first achieved for $n=2$ by Sakai using the Riemann mapping theorem \cite{sakai1981}. (In what follows we will restrict ourselves to the classification of solutions such that the \textit{coincidence set}, that is the set $\{u=0\}$, has non-empty interior (see Remark \ref{emptyinteriorremark})).

\begin{theorem}[Sakai, 1981]
	\label{theorem:sakai}
	Let $n=2$ and let $u$ be a solution of \eqref{obstacleproblem} such that $\{u=0\}$ has non-empty interior. Then $\{u=0\}$ is either a half-plane, ellipse, parabola or a strip.
\end{theorem}

After Sakai's result, the classification of solutions to \eqref{obstacleproblem} with bounded coincidence set was achieved using tools from potential theory for all dimensions $n\geq 2$ in \cite{friedman1986characterization}  (see also \cite{dive1931attraction,lewy1979inversion,di1986bubble} for earlier partial classifications as well as \cite{eberle2021characterizing} for a short proof). On the other hand, the case when the coincidence set is unbounded remained an open problem until very recently. A complete classification for dimensions $n\geq 6$ was given in \cite{eberleduke2022} while the restriction on the dimension was later removed in \cite{eberle2022complete}, thus completing the classification for dimensions $n\geq 3$. 

\begin{theorem}[Eberle, Figalli and Weiss (2022)]
	\label{3dimthm}
	Let $n\geq3$ and let $u$ be a solution of \eqref{obstacleproblem} such that $\{u=0\}$ has non-empty interior. Then $\{u=0\}$ is either a half-space, ellipsoid, paraboloid or a cylinder with ellipsoid or paraboloid as base. 
\end{theorem}

In both \cite{eberleduke2022} and \cite{eberle2022complete}, the focus is showing that solutions with unbounded, non-cylindrical coincidence sets are paraboloids as this implies Theorem \ref{3dimthm} (see for instance \cite[Section 3.2]{eberleduke2022}). Broadly speaking, the general method in both proofs is to match the asymptotic behaviour of $u$ at infinity with that of a solution of \eqref{obstacleproblem} with prescribed coincidence set, and then show that these two solutions are indeed identical.

In the case $n\geq 4$, the method in \cite{eberle2022complete} is to first obtain an expansion of a solution $u$ to \eqref{obstacleproblem} of the form 
\begin{equation}\label{expansion:example}
	u = p(x) + b\cdot x + o(\abs{x}),
\end{equation}
where $p(x) = \lim_{r \to \infty} r^{-2}u(rx)$ is the blow-down and $b \in \R^n$ is some vector that depends on the solution. Then a solution $u_{\text{matched}}$ is constructed so that $\{u_{\text{matched}} = 0\}$ is a paraboloid and $u_{\text{matched}}=p(x) + b\cdot x + o(\abs{x})$. This choice of matching meant that $u-u_{\text{matched}}$ has sublinear growth at infinity so that an application of the Alt-Caffarelli-Friedmann (ACF) functional (see \cite{acf84}) yields that $u=u_{\text{matched}}$.

For $n=3$ however, an expansion as in \eqref{expansion:example} is not possible. In fact, it can be shown that $\strokedint_{B_R} \abs{u-p} \simeq R\log(R)$ when $\{u=0\}$ is a paraboloid and hence in this sense, the case $n=3$ is critical for the method. Nevertheless, a two stage version of the method where the $R\log(R)$ behaviour is carefully removed succeeds in constructing a solution $u_{\text{matched}}$ such that $\{u_{\text{matched}} = 0\}$ is a paraboloid and $u-u_{\text{matched}}$ has sublinear growth at infinity. 

When $n=2$ it can be shown that $\strokedint_{B_R} \abs{u-p} \simeq R^{\frac{3}{2}}$ and so this corresponds to a supercritical case for the method. The purpose of this paper is to classify solutions $u$ to \eqref{obstacleproblem} with unbounded coincidence sets when $n=2$ in the spirit of \cite{eberle2022complete}. In this way, we give a proof of Theorem \ref{theorem:sakai} that avoids complex analysis techniques.  

Apart from unifying the proofs of Theorems \ref{theorem:sakai} and \ref{3dimthm}, it is important to give a proof of Theorem \ref{theorem:sakai} that does not rely on complex analysis since there are closely related problems in two dimensions that complex analysis cannot solve. For instance, we believe that global solutions to the obstacle problem in two dimensional exterior domains cannot be addressed using complex analysis techniques, while the techniques developed in this paper could shed light on how to approach this related problem. 

As was noticed in both \cite{eberleduke2022,eberle2022complete}, and as will be explained later for the case $n=2$, in order to prove Theorem \ref{theorem:sakai}, we only need to classify a certain class of solutions to \eqref{obstacleproblem}. Precisely, we will show the following:
\begin{theorem}\label{maintheorem}
	Let $N=2$ and $u$ be an $x_2$-monotone global solution to the obstacle problem as in Definition \ref{x2mondef} below. Then $\{u=0\}$ is a paraboloid. 
\end{theorem}

Although the general strategy is shared with that in \cite{eberle2022complete}, in the supercritical setting, we must overcome some new difficulties. 

Similar to the case when $n=3$, obtaining an expansion such as \eqref{expansion:example} is not possible. However, it is also not possible to remove the $R^{\frac{3}{2}}$ behaviour from the difference $u-p$. Instead, we can only prove an expansion of the form
\begin{equation}\nonumber
u = p(x) + \alpha v + o(\abs{x}^{\frac{3}{2}}),
\end{equation}
where $v$ is the $\frac{3}{2}$-homogeneous solution to the thin obstacle problem. Unlike in dimensions $n\geq 3$, our proof of the expansion does not rely on a decomposition of the generalised Newtonian potential of the set $\{u=0\}$, but rather on the classification of blow-down limits of the sequence $r^{-\frac{3}{2}}(u-p)(rx)$ as $r\to \infty$. In order to do this, we must first prove the almost monotonicity of the Almgren frequency functional and from this conclude sharp growth estimates of $u-p$.  

The second difficulty we must overcome is that $u-u_{\text{matched}}$ is no longer sublinear at infinity and so the standard ACF functional cannot be used to conclude that $u=u_{\text{matched}}$. In our setting, the difference $u-u_{\text{matched}}$ is only $o(\abs{x}^{\frac{3}{2}})$ and so we must develop a modified ACF functional in two dimensions that converges to zero for functions that are $o(\abs{x}^{\frac{3}{2}})$.

\subsection{Structure of the paper}
In Section \ref{section:notationandpreliminaries} we recall known results regarding solutions to the obstacle problem. In particular, we recall the existence of paraboloid solutions of \eqref{obstacleproblem} and give a generalised Newtonian potential expansion for these solutions. We conclude Section \ref{section:notationandpreliminaries} by recalling important results on solutions to the thin obstacle problem which play an important role in our analysis. In Section \ref{section:modifiedACF} the modified ACF functional is constructed. Section \ref{section:blowdownanalysis} is dedicated to the classification of the blow-down limits of $u-p$ with respect to the frequency scaling. Section \ref{section:coincidence:estimates} uses the blow-down analysis to obtain estimates on the coincidence set while Section \ref{section:sharpgrowth} is dedicated to obtaining sharp growth estimates of $u-p$ by establishing the almost monotonicity of the Almgren frequency for $u-p$. In Section \ref{section:refined:blowdown} we identify paraboloid solutions satisfying $u-u_{\text{matched}}=o(\abs{x}^{\frac{3}{2}})$, and in Section \ref{section:conclusion}, using the modified ACF functional we conclude.

\section{Notation and Preliminaries}\label{section:notationandpreliminaries}
\subsection{Notation}
Throughout this work, a point $x\in \R^n$ will be denoted by $(x_1,...,x_n)$ while $\{e_i\}_{1\leq i \leq n}$ represents the canonical basis vectors of $\R^n$. A ball with centre $x$ and radius $r>0$ will be denoted as $B_r(x)$ and when $x$ is the origin it will be omitted. 

We will denote by $\R^n_{+}$ the upper half space $\{x \in \R^n: x_1 >0\}$ and similarly $\R^{n}_-$ represents $\{x \in \R^n: x_1 < 0\}$. The sets $B_r(x) \cap \R^n_+$ and $B_r(x) \cap \R^n_+$ will be denoted by $B_r(x)^-$ and $B_r(x)^+$  respectively. We will denote the outward unit normal on $B_1$ as $\nu$.

When $A\subset \R^n$, the $n$-dimensional Lebesgue measure of $A$ will be denoted by $\abs{A}$, while $\mathcal{H}^m(A)$ will denote the $m$-dimensional Hausdorff measure. The interior of $A$ will be represented by $A^{\circ}$. Given $B \subset \R^n$ the Hausdorff distance between $A$ and $B$ will be denoted by $\dist(A,B)$. 

Given any $w: \R^n \to \R$ we will denote by $w_r(x)$ the rescaled function $w(rx)$. We will denote $f_+$ the function $\max{\{f,0\}}$ and $f_-$ the function $\max\{-f,0\}$. Functions of the form $\frac{1}{2}(x\cdot e)_+^2$ for some $e \in \mathbb{S}^{n-1}$ will be called \textit{half-space solutions}.

We will denote the \textit{coincidence set} as $\mathcal{C} = \{u=0\}$ and we will refer to $\p \mathcal{C}$ as the \textit{free boundary}. We recall the following properties for solutions of \eqref{obstacleproblem} (see for instance \cite{petrosyan2012regularity}). Any constant that depends only on the dimension will be called universal.
\subsection{Basic properties of global solutions}
We recall the following properties for solutions of \eqref{obstacleproblem} (see for instance \cite{petrosyan2012regularity}). 
\begin{lemma}[Basic Properties]\label{lemma:basic:properties}
	Let $u$ be a solution to \eqref{obstacleproblem}. Then the following hold:
	\begin{enumerate}[i.]
		\item if $0 \in \p \mathcal{C}$ then $u$ is uniformly $C^{1,1}$, that is there exists a universal constant $C$ such that $\norm{D^2u}_{L^{\infty}(\R^2)} \leq C$,
		\item $u$ is convex and hence the coincidence set $\mathcal{C}$ is convex,
		\item and if $\mathcal{C}$ has non-empty interior then $\p \mathcal{C}$ is analytic.
	\end{enumerate}
\end{lemma}
Throughout the rest of this work we will assume that $0 \in \p C$.
\begin{remark}
	\label{emptyinteriorremark}
	If $\{u=0\}$ has empty interior then $\Delta u \equiv 1$ (since $u \in C^{1,1}$) and so Liouville's theorem implies that $u$ must be a quadratic polynomial. 
\end{remark}

We will now recall the subharmonic properties of the difference of two solutions to the obstacle problem.  Again these results can be found in \cite[Lemma 2.13]{eberle2022complete}.
\begin{lemma}\label{differenceenergyinequality}
	Let $u_1,u_2: \R^2\to \R$ be global solutions to the obstacle problem. Then $(u_1-u_2)_+$, $(u_1 -u_2)_-$ and $\abs{u_1-u_2}$ are subharmonic and we have that
	\begin{equation*}
		\strokedint_{B_r} \abs{\nabla (u_1-u_2)}^2 dx \leq \frac{C}{r^2} \strokedint_{B_{2r}} (u_1-u_2)^2 dx.
	\end{equation*}
\end{lemma}

The blow-down limits with respect to quadratic rescaling have been classified in \cite{caffarelli1998obstacle} and this is the content of the following Lemma.

\begin{lemma}[Classification of blow-downs]
	\label{blowdownlimits}
	Let $u$ be a solution to \eqref{obstacleproblem}. Then 
	\begin{equation*}
		\frac{u(rx)}{r^2} \to u_{0} \text{ in }C^{1,\alpha}_{\loc}(\R^2)
	\end{equation*}
	where $u_{0}$ is a solution of \eqref{obstacleproblem}. Moreover, either $u_{0}$ is a non-negative quadratic polynomial or $u_{0}(x) = \frac{1}{2}(x\cdot e)_+^2$ for some $e \in \mathbb{S}^{n-1}$. 
\end{lemma}

\begin{remark}\label{halfspaceremark}
	Observe that if the blow-down limit of a global solution $u$ is a half-space solution, then $u$ must be a half-space solution. Indeed, supposing $u(0) =0$, \cite[Proposition 5.4]{petrosyan2012regularity} implies that $u$ is a monotone function of one variable, that is, there exists some $U \in C^{1,1}(\R)$ and a direction $e \in \mathbb{S}^{n-1}$ such that $u(x)=U(x\cdot e)$ where $U''(t) =\chi_{\{U>0\}}$, $U(0)=U'(0)=0$ and $U'(t)\geq 0$. This implies that $U(t)=\frac12 t_+^2$ and so $u$ is a half-space solution. 
\end{remark}

Since we are interested in classifying global solutions with unbounded coincidence sets, we will recall the following Lemma which establishes the monotonicity of global solutions whose coincidence set contains a line that goes to infinity in some direction \cite{caffarelli2000regularity}. Observe that by convexity, if $\mathcal{C}$ is unbounded then it must contain an infinite ray in some direction.

\begin{lemma}[\cite{caffarelli2000regularity}]\label{monotonicty:global}
	Let $u$ be a solution to \eqref{obstacleproblem} such that $\mathcal{C}$ contains a ray in the $e$-direction. Then either $(\p_eu)_+ \equiv 0$ or $(\p_e u)_- \equiv 0$. 
\end{lemma}

We will now restrict our attention to solutions of \eqref{obstacleproblem} for the case $n=2$. If $\mathcal{C}$ contained an infinite line, say $\{x_1 =0\}$, then $u$ is independent of the $x_2$-direction (see for instance \cite[Lemma 2.7]{eberle2022complete}). Therefore, if $\mathcal{C}$ has non-empty interior and contained $\{x_1 =0\}$, then $\mathcal{C}$ must be a strip in $\R^2$ and the coincidence set is trivially classified. If on the other hand $\mathcal{C}$ has empty interior and contained $\{x_1 =0\}$ then $\Delta u \equiv 1$ and we have that $u(x)= \frac{1}{2}x_1^2$ so that the coincidence set is just the line $\{x_1=0\}$. Therefore we do not need to consider the case when $\mathcal{C}$ contains an infinite line. 

Now if $u$ is monotone, and $\mathcal{C}$ is convex and does not contain an infinite line, we can assume, up to a rotation and translation, that $\mathcal{C} \cap \{x_2< 0\} = \emptyset$. Moreover, since the free boundary is analytic when $\mathcal{C}$ has non-empty interior, we can assume that $\mathcal{C}\cap \{x_2\leq 0\} = \{0\}$. Indeed, if $\mathcal{C} \cap \{x_2\leq0\} $ contained a line segment, then the free boundary would have to be $\{x_2=0\}$ by analyticity, in which case $u$ would be the half-space solution.

Along with Remark \ref{halfspaceremark}, the above considerations imply that we only need to classify a certain class of solutions to \eqref{obstacleproblem}. The focus of this work will therefore be on classifying \textit{$x_2$-monotone solutions} as in the following definition. 

\begin{definition}[$x_2$-monotone solutions]\label{x2mondef}
	A global solution $u$ to the obstacle problem \eqref{obstacleproblem} is an $x_2$-monotone solution if
	\begin{enumerate}[(i)]
		\item the blow-down limit is $p(x)=\frac{1}{2}x_1^2$,
		\item $\p_2 u \leq 0$,
		\item $\mathcal{C} \cap \{x_2\leq 0\} = \{0\}$,
		\item and $\mathcal{C}$ has non-empty interior.
	\end{enumerate}
\end{definition} 

In order to classify the second order blow-down we will need the following Hölder bound for the difference $u-p$ when $n=2$  \cite[Lemma A.5]{franceschini2021c}. It is important to note that this Lemma is only true in higher dimensions if $p(x)$ is kept as $\frac12 x_1^2$. In particular, this is not true for the difference between $u$ and its blow-down if $n\geq 3$. 

\begin{lemma}\label{holderboundlemma}
	There exists some $\beta \in (0,1)$ and a universal constant $C$ so that 
	\begin{equation}
		\label{holderbound}
		\norm{(u-p)_r}_{C^{0,\beta}(B_{1/2})} \leq C\norm{(u-p)_r}_{L^2(B_1)}.
	\end{equation}
\end{lemma}
We provide the proof  of Lemma \ref{holderboundlemma} in Appendix \ref{holderboundapp}. 

\begin{remark}\label{signedremark}
	Two key properties we will use is that $(u-p)\Delta (u-p) \geq 0$ and that $\p_{11}(u-p)\leq 0$. Indeed, 
	\begin{equation*}
		(u-p)\Delta(u-p) = p\chi_{\{u=0\}} \geq 0
	\end{equation*}
	and 
	\begin{equation*}
		\p_{11}(u-p)= \Delta (u-p) - \p_{22} (u-p) = -\chi_{\{u=0\}} - \p_{22} u \leq 0,
	\end{equation*}
	using the convexity of $u$. 
\end{remark}

Finally, we recall that for $w \in C^{1,1}$ the Almgren frequency functional is given by
\begin{equation*}
	\phi(r,w) = r\frac{\int_{B_r} \abs{\nabla w}^2}{\int_{\p B_r }w^2d\sigma}
\end{equation*}
and we observe that we have the following upper bound (see \cite[Proposition 4.1]{eberleduke2022}).
\begin{proposition}[Frequency Estimate]
	For all $r > 0$ we have that
	\begin{equation}\label{frequencybound}
		\phi(r,u-p) \leq 2
	\end{equation}
\end{proposition}

\subsection{Paraboloid solutions}
For each $\gamma > 0$ we define the paraboloid $P_{\gamma} = \{-\gamma^{1/2}\sqrt{y_2}\leq y_1\leq\gamma^{\frac{1}2}\sqrt{y_2}, y_2\geq 0\}=\gamma \{-\sqrt{y_2}\leq y_1 \leq \sqrt{y_2},y_2 \geq 0\}$. We will denote $P=P_1$ and so $P_{\gamma}=\gamma P$. 
The following Proposition has already been shown with complex analysis techniques in \cite{sakai1981} as well as a potential theory approach in \cite{shahgholian1992quadrature}. 
\begin{proposition}[Existence of Paraboloid Solutions]\label{paraboloidsolutions}
	Given $\gamma >0$ there exists a solution to \eqref{obstacleproblem} denoted by $u_{\gamma P}$ such that $\{u_{\gamma P}=0\} = P_{\gamma}$. 
\end{proposition}

\begin{proposition}\label{blowdonw2d}
	For each $\gamma > 0$ we have that
	\begin{equation}\label{paraboloidblowdown}
		\lim_{r\to\infty} \frac{u_{\gamma P}(rx)}{r^2} = \frac{1}{2}x_1^2.
	\end{equation}
\end{proposition}
\begin{proof} 
	It is clear by Remark \ref{halfspaceremark} that $u_0 := \lim_{r \to \infty} r^{-2}u_{\gamma P}(r\cdot)$ is a quadratic polynomial. Since $\{x_1 =0, x_2\geq 0\} \subset \{u_0=0\}$ we must have that $u_0$ is monotone in the $x_2$-direction by Lemma \ref{monotonicty:global} and hence $u_0$ cannot contain multiples of $x_1x_2$ or $x_2^2$. Since $n=2$ this means that $u_0(x)=\frac{1}{2}x_1^2$.
\end{proof}

It is shown in \cite{eberle2022complete} that solutions to \eqref{obstacleproblem} with $n\geq 3$ can be expressed in terms of a generalised Newtonian potential. For our purposes it is enough to establish this result for paraboloid solutions. 

Defining
\begin{equation*}
	G(x,y)=\log(\abs{x-y})-\log(\abs{y})+\frac{x\cdot y}{\abs{y}^2},
\end{equation*}
the generalised Newtonian potential for a set $M\subset \R^2$ is then
\begin{equation*}
	V_M(x)= -\frac{1}{2\pi} \int_M G(x,y)dy.
\end{equation*}
We have the following important scaling property.
\begin{lemma}\label{scaling}
	For each $\gamma > 0$ we have that 
	\begin{equation*}
		V_M(\gamma x) = \gamma^2 V_{\frac{1}{\gamma}M}(x)
	\end{equation*}
\end{lemma}
\begin{proof}
	We first observe that
	\begin{align*}
		G(\gamma x, y)
		&=\log(\abs{\gamma x-y})-\log(\abs{y})+\frac{\gamma x\cdot y}{\abs{y}^2}\\
		&= \log(\abs{x-\frac{y}{\gamma}}) + \log(\gamma) -\log(\abs{\frac{y}{\gamma}}) - \log(\gamma)+\frac{x\cdot \frac{y}{\gamma}}{\abs{\frac{y}{\gamma}}^2}\\
		&= G(x,\frac{y}{\gamma}).
	\end{align*}
	Using this scaling we have that
	\begin{align*}
		V_M(\gamma x)
		&= -\frac{1}{2\pi}\int_M G(x,\frac{y}{\gamma}) dy \\
		&= -\frac{1}{2\pi}\int_{\frac{1}{\gamma}M} G(x,z) \gamma^2dz \\
		&= \gamma^2 V_{\frac{1}{\gamma}M}(x).
	\end{align*}
\end{proof}
We can now state the expansion for paraboloid solutions.
\begin{proposition}\label{expansionthm}
	Let $u_{\gamma P}$ be a paraboloid solution of \eqref{obstacleproblem} in $\R^2$. Then 
	$$u_{\gamma P}(x)= p(x) + V_{\gamma P}(x)$$ for all $x \in \R^2$. 
\end{proposition}
The proof of Proposition \ref{expansionthm} is deferred to Appendix \ref{section:appendix:expansion}. 

\subsection{The thin obstacle problem}
Finally, an important role will be played in our analysis by homogeneous global solutions to the thin obstacle problem as these appear as the second blow-down limit. We recall here that for $R>0$, $v$ is a solution to the thin obstacle problem with zero obstacle in $B_R \subset \R^2$ if
	\begin{align}\label{thinobstacleglobal}
	\begin{split}
	v\Delta v &= 0 \text{ and } \Delta v \leq 0 \text{ in } B_R\\
	v &\geq  0 \text{ on } \{x_1=0\}\cap B_R \text{ and}, \\
	\Delta v &= 0 \text{ on } B_R\backslash \{x_1=0\}.
	\end{split}
	\end{align}
If $v$ solves \eqref{thinobstacleglobal} in $\R^2$ we will call $v$ a global solution to the thin obstacle problem. 

We now recall the following optimal regularity for solutions to the thin obstacle problem \cite[Theorem 9.13]{petrosyan2012regularity}
\begin{theorem}\label{optimalregularity}
	Let $v$ be a solution of \eqref{thinobstacleglobal}.
	There exists some constant $C=C(R)$ such that 
	\begin{equation*}
		\norm{v}_{C^{1,\frac{1}{2}}(\overline{B_{\frac{R}{2}}^{\pm}})}  \leq C \norm{v}_{L^2(B^{\pm}_R)}.
	\end{equation*}
\end{theorem}

We have the monotonicity of the Almgren frequency for solutions $v$ of \eqref{thinobstacleglobal} \cite[Theorem 9.4]{petrosyan2012regularity}.
\begin{proposition}\label{almgrenmonotonicitythin}
	If $v$ is a solution of \eqref{thinobstacleglobal} then $\phi(r,v)$ is monotone non-decreasing for $0 < r < R$. Moreover $\phi(r,v) \equiv \kappa$ for $0 < r < R$ if and only if $v$ is homogeneous of degree $\kappa$ in $B_R$.
\end{proposition}

The homogeneous global solutions to \eqref{thinobstacleglobal} are classified and we summarise these results from \cite[Section 9.4.1]{petrosyan2012regularity} in the following Lemma. 
\begin{lemma}\label{thinclassification}
	Let $v$ be a $\kappa$-homogeneous global solution of \eqref{thinobstacleglobal}. Then the following hold.
	\begin{enumerate}[i.]
		\item The possible values of $\kappa$ are $1, 2m, 2m-\frac{1}{2}$ or  $2m +1, m \in \N$;
		\item if $\kappa =2m$ then $v$ must be a $2m$-homogeneous harmonic polynomial;
		\item if $\kappa = 2m-\frac{1}{2}$ then $v = Re(x_1 + \abs{x_2}i)^{2m-\frac12}$ up to a constant,
		\item and if $\kappa =1$ or $\kappa = 2m + 1$ then $v = Im(x_1 + \abs{x_2}i)^{2m+1}$ up to a constant.
	\end{enumerate}
\end{lemma}

\section{Modified ACF functional}\label{section:modifiedACF}
In this section we establish a modified ACF functional in $2$-dimensions. 

\begin{theorem}\label{prop:modACF}
	Suppose for $i=1,2,3$, $v_i \in C^{1,1}(B_2)$ satisfy $v_i\Delta v_i \geq 0$ and have supports in $\Omega_i$ pairwise disjoint. Moreover suppose that $v_i =  0$ on $\p (\Omega_i \cap B_1)$. Then for each $\beta\leq 9$, the quantity
	\begin{equation}
		\nonumber
		 \Phi(r) = \frac{1}{r^{\beta}}\prod_{i=1}^3\int_{B_r} \abs{\nabla v_i}^2 dx
	\end{equation}
	is monotone non-decreasing in $r$,  $0<r<1$.
\end{theorem}

\begin{proof}
	Denoting $\frac{d}{dr}\Phi(r)$ as  $\Phi'(r)$, we have
	\begin{align*}
		\Phi'(r) = -\beta r^{-1}\Phi(r) + r^{\beta}\sum_{i=1}^{3} \int_{\p B_r} \abs{\nabla v_i}^2 d\sigma \prod_{j\neq i} \int_{B_r} \abs{\nabla v_j}^2.
	\end{align*}
	Dividing through by $r^{-1}\Phi(r)$ we have 
	\begin{align*}
		r \frac{\Phi'(r)}{\Phi(r)} = -\beta + r\sum_{i=1}^{3} \frac{\int_{\p B_r} \abs{\nabla v_i}^2 d\sigma }{\int_{B_r} \abs{\nabla v_i}^2},
	\end{align*}
	so that by rescaling, we just need to show that the quantity 
	\begin{equation}\nonumber
		P(\beta)  := -\beta + \sum_{i=1}^{3} \frac{\int_{\p B_1} \abs{\nabla v_i}^2 d\sigma }{\int_{B_1} \abs{\nabla v_i}^2}
	\end{equation}
	is non-negative. Since $v_i\Delta v_i \geq 0$, integrating by parts yields
	\begin{align*}
		\int_{B_1} \abs{\nabla v_i}^2 = - \int_{B_1} v_i \Delta v_i + \int_{\p B_1} v_i (v_i)_{\rho} d\sigma \leq \int_{\p B_1} v_i (v_i)_{\rho} d\sigma,
	\end{align*}
	where $(v_i)_{\rho}$ denotes the exterior radial derivative along $\p B_1$. Applying then H\"older's inequality we have
	\begin{align*}
		\int_{\p B_1} v_i (v_i)_{\rho} d\sigma \leq \left(\int_{\p B_1} v_i^2 d\sigma\right)^{1/2}\left(\int_{\p B_1}(v_i)_{\rho}^2 d\sigma\right)^{1/2}.
	\end{align*}
	On the other hand, Young's inequality yields
	\begin{equation*}
		\int_{\p B_1} \abs{\nabla v_i}^2 \geq 2 \left(\int_{\p B_1}(v_i)_{\rho}^2 d\sigma\right)^{1/2} \left(\int_{\p B_1}(v_i)_{\theta}^2 d\sigma\right)^{1/2},
	\end{equation*}
	where $(v_i)_{\theta}$ is the tangential derivative of $v_i$ along $\p B_1$. We can therefore estimate $P(\beta)$ from below and obtain 
	\begin{equation}\label{ACFproof:quantity2}
		P(\beta) \geq  2\left(\sum_{i=1}^{3} \frac{\left(\int_{\p B_1} \abs{(v_i)_{\theta}}^2d\sigma \right)^{\frac12} }{\left(\int_{\p B_1} \abs{v_i}^2d\sigma\right)^{\frac12}} - \frac{\beta}{2}\right).
	\end{equation}
	We must now minimise the sum appearing in the right-hand of \eqref{ACFproof:quantity2}. To this end we let $\Gamma_i = \p B_1 \cap \Omega_i$ and we recall some remarks from \cite[Proof of Theorem 12.1]{caffarelli2005geometric}. Observe that
	\begin{equation*}
		\frac{\int_{\p B_1} \abs{(v_i)_{\theta}}^2 d\sigma}{\int_{\p B_1} \abs{v_i}^2 d\sigma} \geq \inf_{v\in H^1_0(\Gamma_i)} \frac{\int_{\p B_1} \abs{(v)_{\theta}}^2d\sigma }{\int_{\p B_1} \abs{v}^2d\sigma}
	\end{equation*}
	where the right-hand side is the first eigenvalue of the domain $\Gamma_i$. Among all possible domains with measure $\abs{\Gamma_i}$, this infimum is attained when the domain is a connected arc\footnote{Indeed, if $\Gamma_i$ was not connected then symmetrising $v$ cannot increase the quotient by the P\'oly-Szeg\H{o} inequality.}. Since $\Gamma_i \cap \Gamma_j = \emptyset$ if $i\neq j$, and since larger arcs have smaller eigenvalues, we have that the sum in \eqref{ACFproof:quantity2} will be minimised if the $\Gamma_i$ are three disjoints arcs whose union is the circle. 
	
	Supposing these three arcs have lengths $2\pi\alpha$, $2\pi\gamma$ and $2\pi(1-\alpha-\gamma)$ for some $\alpha,\gamma \in (0,1)$ satisfying $0<\alpha + \gamma <1$, we find that the three principle eigenfunctions are $\sin(\frac{\theta}{2\alpha})$, $\sin(\frac{\theta}{2\gamma})$ and $\sin(\frac{\theta}{2(1-\alpha-\gamma)})$. We therefore have that
	\begin{align*}
		 \sum_{i=1}^{3} \frac{\left(\int_{\p B_1} \abs{(v_i)_{\theta}}^2 d\sigma\right)^{\frac12} }{\left(\int_{\p B_1} \abs{v_i}^2d\sigma\right)^{\frac12}}  \geq \frac{1}{2\alpha} + \frac{1}{2\gamma} + \frac{1}{2(1-\alpha-\gamma)} \geq \frac{9}{2}
	\end{align*}
	where in the last step we minimised $\frac{1}{2\alpha} + \frac{1}{2\gamma} + \frac{1}{2(1-\alpha-\gamma)}$ over the domain $\{0< \alpha < 1\} \cap \{0< \gamma < 1\} \cap \{0< \alpha + \gamma < 1\}$. This completes the proof as we then have
	\begin{equation*}
		 P(\beta) \geq  9 - \beta,
	\end{equation*}
	which is non-negative if $\beta \leq 9$. 
\end{proof}

\section{Blow-down analysis}\label{section:blowdownanalysis}

We begin by defining the family of rescalings
$$\tilde{w}_r = \frac{w_r}{\norm{w_r}_{L^2(\p B_1)}}.$$
Thanks to the frequency estimate in Proposition \ref{frequencybound} there exists a universal constant $C$ such that
\begin{equation}
	\norm{\tilde{w}_r}_{W^{1,2}(B_1)} \leq C \text{\hspace{2mm} for all } r>0.
\end{equation}
The following Proposition shows that this can be improved to uniform $W^{1,2}(B_R)$-bounds for each $R>1$.

\begin{proposition}\label{uniformW12}
	For each $R >1$ we have 
	\begin{equation*}
		\norm{\tilde{w}_r}_{W^{1,2}(B_R)} \leq C(R)
	\end{equation*}
	 for all $r > 1$.
\end{proposition}
\begin{proof}
	We follow \cite[Proposition 4.2]{eberleduke2022}. First observe that from the frequency estimate \eqref{frequencybound} for $w$ at radius $Rr$ for every $r>0$, we obtain
	\begin{equation*}
		\int_{B_R} \abs{\nabla w_r}^2 \leq \frac{2}{R} \int_{\p B_R} w_r^2 d\sigma.
	\end{equation*}
	Consequently for every $R>1$ and every $r >0$ we have that 
	\begin{equation*}
		\norm{\tilde{w}_r}_{W^{1,2}(B_R)} \leq C \frac{\int_{\p B_R} w_r^2d\sigma}{\int_{\p B_1}w_r^2d\sigma}.
	\end{equation*}
	Now suppose that for some $R>1$ there exists a sequence $(r_k)_{k\in\N}$ such that $r_k \to \infty$ and
	\begin{equation}\label{uniformW12proof:assumption}
		\frac{\int_{\p B_R} w_{r_k}^2d\sigma}{\int_{\p B_1}w_{r_k}^2d\sigma} \to + \infty \hspace{2mm} \text{as } k\to\infty
	\end{equation}
	 Defining the sequence
	\begin{equation*}
		\bar{w}_{r_k}= \frac{w_{r_k}}{\norm{w_{r_k}}_{L^2(\p B_R)}},
	\end{equation*}
	we have that
	\begin{equation*}
		\int_{\p B_R} \abs{\bar{w}_{r_k}}^2d\sigma =1,
	\end{equation*}
	so that by the frequency estimate \eqref{frequencybound} we obtain
	\begin{equation}\label{uniformW12proof:gradbound}
		\int_{B_R} \abs{\nabla \bar{w}_{r_k}}^2 \leq 2 R^{-1}.
	\end{equation}
	Along with Lemma \ref{holderboundlemma}, \eqref{uniformW12proof:gradbound} implies that there exists some $\bar{w} \in W^{1,2}(B_R)$ such that, up to taking a subsequence which we do not relabel, we have
	\begin{align*}
		&\bar{w}_{r_k} \to \bar{w} \text{\hspace{2mm} weakly in } W^{1,2}(B_R),\\
		&\bar{w}_{r_k} \to \bar{w} \text{\hspace{2mm} in } L^{2}(\p B_R),\\
		&\bar{w}_{r_k} \to \bar{w} \text{\hspace{2mm}  in } C^{0}_{\loc}(B_R).
	\end{align*}
	Moreover, since $\Delta w_{r_k} = -\chi_{\{u_{r_k}=0\}} < 0$ we have for each $\alpha \in (0,1)$ and each $\eta \in C^{\infty}_c(B_R)$ satisfying $\eta \geq 0$ and $\eta \equiv 1$ on $B_{\alpha R}$, that 
	\begin{align*}
		\int_{B_{\alpha R}} \abs{\Delta \bar{w}_{r_k}} \leq   \int_{B_{R}} \eta \abs{\Delta \bar{w}_{r_k}} = - \int_{B_R} \eta \Delta \bar{w}_{r_k} \leq C(\alpha) \norm{\nabla \bar{w}_{r_k}}_{L^2(B_R)}.
	\end{align*}
	Hence $\Delta \bar{w}_{r_k}$ converges locally to $\Delta \bar{w}$ as measures and in particular, $\Delta \bar{w}$ is a non-positive measure supported on $\{t e^2: t\geq 0\} \cap B_{\theta R}$ for each $\theta \in (0,1)$.\\
	We will now show that $\bar{w} \equiv 0$ on $B_1$. Firstly, the locally uniform convergence implies that $\bar{w} =0 $ on  $\{t e^2: 0 \leq t \leq 1\}$. Moreover, by \eqref{uniformW12proof:assumption}, we also have that 
	\begin{equation*}
		\int_{\p B_1} \bar{w}_{r_k}^2d\sigma = \frac{\int_{\p B_1} w_{r_k}^2d\sigma}{\int_{\p B_R} w_{r_k}^2d\sigma} \to 0,
	\end{equation*}
	and so $\bar{w} \equiv 0$ on $\p B_1$. Since $\bar{w}$ is harmonic in the set $\Omega=B_1 \backslash \{t e^2: 0 \leq t \leq 1\}$ while being zero on $\p \Omega$, we must have that $\bar{w} \equiv 0$ in $B_1$.
	
	Now, since $\bar{w}$ is analytic in $B_R \backslash  \{t e^2: 0 \leq t \leq R\}$ and zero on  $B_1$, we conclude that $\bar{w} \equiv 0$ in $B_R$. This contradicts the fact that 
	$$\int_{\p B_R} \bar{w}^2 d\sigma= \lim_{k\to\infty} \int_{\p B_R}\bar{w}_{r_k}^2d\sigma=1.$$
\end{proof}

\begin{proposition}\label{blowdownanalysis}
	For each sequence $(r_k)_{k\in \N}$ such that $r_k \to \infty$ as $k \to \infty$, there exists a subsequence, which we do not relabel, such that
	\begin{equation*}
		\tilde{w}_{r_k} \to v_0 \text{  in } W^{1,2}_{\loc}\cap C^0_{\loc}(\R^2),
	\end{equation*}
	where $v_0$ is a global solution to the thin obstacle problem \eqref{thinobstacleglobal} (which may depend on the sequence) satisfying 
	\begin{equation}
		\label{norm1blowdown}
		\norm{v_0}_{L^2(\p B_1)} = 1,
	\end{equation}
	\begin{equation}\label{supportmeasure}
		\supp(\Delta v_0) \subset \{x_1 = 0\} \cap \{x_2 \geq 0\},
	\end{equation}
	\begin{equation}\label{nonpositivemeasure}
	\Delta v_0\leq 0,
	\end{equation}
	\begin{equation}
		\label{monotonicty:blowdown1}
		\p_2 v_0 \leq 0,
	\end{equation}
	and
	\begin{equation}
		\label{nonpositive:blowdown1}
		v_0 \leq 0 \hspace{2mm} \text{ in } \{x_2 \geq 0\}.
	\end{equation}
	Moreover,
		\begin{equation}
		\label{frequency:blowdown:bound}
		\phi(R,v_0) \leq 2
	\end{equation}
	for every $R>0$. 
\end{proposition}
\begin{proof}
	\textit{Step 1: Convergence to a global solution of the thin obstacle problem:}\\
	Fixing $R>1$ we have by Proposition \ref{uniformW12} and Lemma \ref{holderboundlemma}, up to a subsequence which we do not relabel, that 
	\begin{align*}
		&\tilde{w}_{r_k} \to v_0 \text{\hspace{2mm} weakly in } W^{1,2}(B_R),\\
		&\tilde{w}_{r_k} \to v_0 \text{\hspace{2mm} in } L^{2}(\p B_R),\\
		&\tilde{w}_{r_k} \to v_0 \text{\hspace{2mm}  in } C^{0}_{\loc}(B_R),
	\end{align*}
	for some $v_0 \in W^{1,2}(B_R)$.  

	As in the proof of Proposition \ref{uniformW12} we have that $\Delta v_0$ is a non-positive measure supported on $\{te^2: 0 \leq t \leq \frac R 2\}$ while $v_0 = 0$ on $\{te^2: 0 \leq t \leq \frac R 2\}$. Moreover, $v_0\geq 0 $ on $\{x_1 =0\} \cap B_{\frac{R}{2}}$.
	We therefore conclude that 
	\begin{align}\nonumber
		\begin{split}
		v_0\Delta v_0 = 0 &\text{ and } \Delta v_0 \leq 0 \text{ in } B_{\frac R2} \\
		v_0 \geq 0 &\text{ on } \{x_1 =0\} \cap B_{\frac{R}{2}} \\
		\Delta v_0 = 0 &\text{ on } B_{R/2}\backslash \left(\{x_1 =0 \} \cap \{x_2\geq 0\}\right),
		\end{split}
	\end{align}
	that is, $v_0$ is a solution to the thin-obstacle problem in $B_{\frac R2}$ with thin obstacle $\{x_1 =0\} \cap B_{\frac{R}{2}}$.\\
	This establishes weak $W^{1,2}_{\loc}(\R^2)$ and $C^0_{\loc}(\R^2)$ convergence to some $v_0$ that is a global solution of \eqref{thinobstacleglobal}. Furthermore, \eqref{norm1blowdown}, \eqref{supportmeasure} and \eqref{nonpositivemeasure} are immediate, while \eqref{frequency:blowdown:bound} follows from Proposition \ref{frequencybound}. Moreover, since $\p_2 w = \p_2 u \leq 0$ we immediately obtain \eqref{monotonicty:blowdown1} by the weak convergence in $W^{1,2}_{\loc}$. Finally we recall that $\p_{11}w \leq 0$ by Remark \ref{signedremark}, and hence $w\leq 0$ in $\{x_2 \geq 0\}$ (since $w=-\frac{1}{2} x_1^2$ on $\{u=0\}$ and $\{u=0\}$ has non-empty interior). This together with the $C^0_{\loc}$ convergence proves \eqref{nonpositive:blowdown1}.
	
	\textit{Step 2: Strong $W^{1,2}_{\loc}$ convergence:}\\
	In this step we will improve the weak $W^{1,2}(B_R)$ convergence to strong $W^{1,2}(B_{\frac{R}{2}})$ convergence. To this end we let $(\eta_j)_{j\in \N}$ be a smooth approximation of $\chi_{B_{\frac{R}{2}}}$ in $B_R$, i.e. $\eta_j \in C^{\infty}_c(B_R)$, $\eta_j \equiv 1$ in $B_{\frac{R}{2}}$ and $\eta_j \to \chi_{B_{\frac{R}{2}}}$ pointwise. We then observe that since $w\Delta w \geq 0$ we obtain
	\begin{align*}
		\int_{B_{\frac{R}{2}}} \abs{\nabla \tilde{w}_{r_k}}^2
		&\leq \int_{B_R} \eta_j\abs{\nabla \tilde{w}_{r_k}}^2\\
		&= -\int_{B_R} \left(\tilde{w}_{r_k} \nabla \eta_j \cdot \nabla \tilde{w}_{r_k} + \eta_j \tilde{w}_{r_k} \Delta \tilde{w}_{r_k}\right) \\
		&\leq  -\int_{B_R} \tilde{w}_{r_k} \nabla \eta_j \cdot \nabla \tilde{w}_{r_k}.
	\end{align*}
	Now passing the limit as $k\to \infty$ we find that
	\begin{align*}
		\limsup_{k\to\infty} \int_{B_{\frac{R}{2}}} \abs{\nabla \tilde{w}_{r_k}}^2 
		&\leq -\int_{B_R}  v_0\nabla \eta_j \cdot \nabla v_0\\
		& = \int_{B_R} \left(\eta_{j} \abs{\nabla v_0}^2 + \eta_j v_0 \Delta v_0\right)\\
		&= \int_{B_R} \eta_{j} \abs{\nabla v_0}^2.
	\end{align*}
	Now passing $j \to \infty$ we obtain 
	\begin{equation*}
		\limsup_{k\to\infty} \int_{B_{\frac{R}{2}}} \abs{\nabla \tilde{w}_{r_k}}^2 \leq \int_{B_{\frac{R}{2}}}\abs{\nabla v_0}^2,
	\end{equation*}
	so that by lower semicontinuity of the Dirichlet energy we have
	\begin{equation*}
		\lim_{k\to\infty} \int_{B_{\frac{R}{2}}} \abs{\nabla \tilde{w}_{r_k}}^2 = \int_{B_{\frac{R}{2}}}\abs{\nabla v_0}^2.
	\end{equation*}
	
	Since for each $B_R \subset \R^2$ we have that $W^{1,2}(B_R) \subset W^{1,\frac32}(B_R)$ by H\"older's inequality, and $W^{1,\frac32}(B_R)$ embeds compactly in $L^2(B_R)$ \cite[Section 5.7, Theorem 1]{evanspde}, the strong $W^{1,2}_{\loc}$ convergence follows.
	
	\end{proof}
	At this point it is difficult to conclude precisely what the blow-down is, or even whether the blow-down obtained depends on the chosen blow-down sequence or not. To overcome this issue, we will now blow-down the blow-down limits obtained in Proposition \ref{blowdownanalysis} and classify these first. To this end we define, in polar coordinates, the function
	\begin{equation}\label{wsin}
		\hat{v}_{3/2}(r,\theta) = -r^{3/2}\cos(\theta)\norm{r^{3/2}\cos(\theta)}_{L^2(\p B_1)}^{-1}
	\end{equation}
	Note that $\hat{v}_{3/2}$ is the $3/2$-homogeneous global solution to the thin obstacle problem with coincidence set $\{x_1 =0\} \cap \{x_2 \geq 0\}$ and $\norm{\hat{v}_{3/2}}_{L^2(\p B_1)}=1$. 
	\begin{proposition}
		\label{blowdown2}
		Given a sequence $(r_k)_{k\in \N}$ going to infinity and $v_0$ the corresponding blow-down limit obtained in Proposition \ref{blowdownanalysis}, we define the family of rescalings
		\begin{equation*}
			(\tilde{v}_{0})_s= \frac{v_0(s\cdot)}{\norm{(v_{0})_s}_{L^2(\p B_1)}}.
		\end{equation*}
		For each sequence $(s_k)_{k\in \N}$ such that $s_k \to \infty$ as $k \to \infty$, we have up to a subsequence which we do not relabel, that
		\begin{equation*}
		(\tilde{v}_{0})_{s_k} \to \hat{v}_{3/2}  
		\end{equation*}
		weakly in $W^{1,2}_{\loc}(\R^2)$ and strongly in  $C^1_{\loc}(\overline{\R^2_{\pm}})$.
	\end{proposition}
	
	\begin{proof} 
		Since the proof is carried out along the same lines as Propositions \ref{uniformW12} and \ref{blowdownanalysis} we will only point out the differences. \\
		\textit{Step 1: Uniform $W^{1,2}(B_R)$ bounds for each $R>1$:}\\
		This step is the same as Proposition \ref{uniformW12} except we do not obtain $C^0_{\loc}$ convergence from Lemma \ref{holderboundlemma}. Rather, from Theorem \ref{optimalregularity} applied in $\overline{B_{R}^{\pm}}$, we obtain convergence in $C^1_{\loc}(\overline{\R^2_{\pm}})$. \\
		\textit{Step 2: Convergence to some blow-down limit:}\\
		Arguing as in Step 1 of Proposition \ref{blowdownanalysis} (again with Lemma \ref{holderboundlemma} replaced by Theorem \ref{optimalregularity}) we obtain weak $W^{1,2}_{\loc}(\R^2)$, strong $L^2(\p B_1)$ and strong $C^1_{\loc}(\overline{\R^2_{\pm}})$ convergence to some limit $v_{00} \in W^{1,2}_{\loc}(\R^2)$. Moreover,
		\begin{equation}
			\label{norm1blowdown2}
			\norm{v_{00}}_{L^2(\p B_1)} =1 ,
		\end{equation}
		\begin{equation}\label{supportmeasure2}
			\supp(\Delta v_{00}) \subset \{x_1 = 0\} \cap \{x_2 \geq 0\},
		\end{equation}
		\begin{equation}\label{nonpositivemeasure2}
			\Delta v_{00} \leq 0,
		\end{equation}
		\begin{equation}
		\label{monotonicty:blowdown2}
		\p_2v_{00} \leq 0,
		\end{equation}
		\begin{equation}
			\label{nonpositive:blowdown2}
			v_{00}  \leq 0 \hspace{2mm} \text{ in } \{x_2 \geq 0\},
		\end{equation}
		and
		\begin{equation}
			\label{frequency:blowdown2:bound}
			\phi(R,v_{00} ) \leq 2 \hspace{5mm} \text{for all } R>0,
		\end{equation}
		are direct consequences of \eqref{norm1blowdown}, \eqref{supportmeasure}, \eqref{nonpositivemeasure}, \eqref{monotonicty:blowdown1}, \eqref{nonpositive:blowdown1} and \eqref{frequency:blowdown:bound}.
		
		\textit{Step 3: Classification of the limits $v_{00}$:}\\
		The first observation is that $v_{00}$ is homogeneous. Indeed using the monotonicity of the Almgren frequency, Proposition \ref{almgrenmonotonicitythin}, as well as the convergence in $C^1_{\loc}(\overline{\R^2_{\pm}})$, we see that for each $R>0$,
		\begin{equation*}
			\phi(R, v_{00} ) = \lim_{k\to \infty} \phi(R s_k,v_0 ) = \phi(+\infty,v_0).
		\end{equation*}
		In this last step we used the monotonicity of $\phi(\rho,v_0)$ in $\rho$ (Proposition \ref{almgrenmonotonicitythin}) and \eqref{frequency:blowdown2:bound} to see that the limit $\phi(+\infty,v_0)$ is well defined. Using Proposition \ref{almgrenmonotonicitythin} and \eqref{frequency:blowdown2:bound} again, we see that $v_{00}$ is a homogeneous function and $\phi(R,v_{00}) \leq 2$ for all $R>0$. By Lemma \ref{thinclassification}, we then have that the frequency is either $2, \frac32$ or $1$. 
		
		If the frequency was $2$ then, by Lemma \ref{thinclassification}, $v_{00}$ would be a $2$-homogeneous harmonic polynomial vanishing along $\{x_1=0\}$. However, in two dimensions the only 2-homogeneous harmonic polynomials satisfying this condition are multiples of $x_1x_2$. Clearly, \eqref{monotonicty:blowdown2} excludes $ v_{00}$ from being a multiple of $x_1x_2$ and so the frequency cannot be $2$.
		
		If on the other hand the frequency was $1$, then again by Lemma \ref{thinclassification}, $v_{00}$ would be a $1$-homogeneous harmonic polynomial in both $\{x_1 \geq 0\}$ and $\{x_1 \leq 0\}$ vanishing along $\{x_1=0\}$. In this case, taking into account \eqref{nonpositivemeasure2}, the options are $\pm x_1$ or $-\abs{x_1}$ which both cannot occur. Indeed $\pm x_1$ violates  \eqref{nonpositive:blowdown2}, while $-\abs{x_1}$ does not satisfy \eqref{supportmeasure2}. 
		
		Since the blow-down cannot be identically zero by \eqref{norm1blowdown2}, we conclude that $v_{00}$ is the $3/2$ homogeneous solution $\hat{v}_{3/2}$ defined in \eqref{wsin}.
		\end{proof}

\begin{proposition}
\label{frequencyconvergesprop}
For each $\eps > 0$ there exists an $r_1(\eps)$ such that 
\begin{equation}
	\label{frequencyconverges}
	\frac{3}{2} - \eps \leq \phi(r,w) \leq \frac{3}{2} + \eps
\end{equation}
for $r\geq r_1(\eps)$. 
\end{proposition}
\begin{proof}
	For each $\eps >0$ we first fix $s_0=s_0(\eps) \geq 1$ large enough so that
	\begin{equation}\label{s0def}
		\abs{\phi(s,v_0) - \frac32} < \frac{\eps}{2}
	\end{equation}
	for every $s\geq s_0$. We will show the following claim.
	\begin{claim}{}
	There exists an $r_0=r_0(\eps) \geq 1$ such that 
	\begin{equation}
		\label{frequencyconverges:claim}
		\frac32 - \eps < \phi(s_0r,w)  < \frac32+ \eps
	\end{equation}
	for every $r\geq r_0$.
	\end{claim}
	The Proposition then follows from the claim by taking $r_1(\eps) = s_0r_0$. 
	\begin{proof}[Proof of Claim:]
		We will prove that the lower bound in \eqref{frequencyconverges:claim} must hold, and then an identical argument can be used to show that the upper bound must hold too. Suppose that for some $\eps >0$ the lower bound in \eqref{frequencyconverges:claim} does not hold. Then for every $k \in \N$ there exists $r_k \geq k$ such that
		\begin{equation}\nonumber
			\phi(s_0r_k,w) < \frac{3}{2} - \eps.
		\end{equation}
		Then letting $k\to\infty$ we find that 
		\begin{equation}\nonumber
			\phi(s_0,v_0) \leq \frac{3}{2} - \eps,
		\end{equation}
		which is a contradiction to \eqref{s0def}.
	\end{proof}
\end{proof}
Proposition \ref{frequencyconvergesprop} tells us that $\lim_{r\to\infty} \phi(r,w) = \frac{3}{2}$ and so we can now use this to classify the blow-down limits obtained in Proposition \ref{blowdownanalysis}.

\begin{corollary}\label{corollary:blowdown:homogeneous}
	Given a sequence $(r_k)_{k\in \N}$ going to infinity and $v_0$ the corresponding blow-down limit obtained in Proposition \ref{blowdownanalysis} we have that $v_0 = \hat{v}_{3/2}$.
\end{corollary}
\begin{proof}
	By Proposition \ref{frequencyconvergesprop} we now have for each $R>0$ that
	\begin{equation*}
		\phi(R, v_0) = \lim_{k\to \infty} \phi(R r_k,w) = \frac{3}{2},
	\end{equation*}
	so that $v_0$ is $\frac32$-homogeneous. Arguing as in Step 3 of Proposition \ref{blowdown2} we conclude that $v_0 = \hat{v}_{3/2}$.
\end{proof}

\section{Growth estimates for the coincidence set}\label{section:coincidence:estimates}

We will begin with the following $3/2$-order doubling.  

\begin{lemma}\label{doublinglemma}
	Given $\lambda \in (0,1)$ there exists an $\bar{r}=\bar{r}(\lambda)$ so that 
	\begin{equation}\label{doubling}
		\norm{w_r}_{L^2(\p B_1)} \leq C(\lambda)r^{\frac{3}{2} + \lambda}
	\end{equation}
	for all $r\geq \bar{r}(\lambda)$. As a consequence,
	\begin{equation}\label{3/2growth}
		C \cap \{y_2 \geq \bar{r}(\lambda)\} \subset \{\abs{y_1}^2 \leq y_2^{\frac{3}{2}+\lambda}\}
	\end{equation}
\end{lemma}

\begin{proof}
	Firstly \eqref{3/2growth} follows from \eqref{doubling} as in \cite[Proposition 5.1]{eberleduke2022}. 
	To obtain \eqref{doubling} we observe that as $r\to \infty$ we have
	\begin{align*}
		\frac{\int_{\p B_1} w_{2r}^2d\sigma}{\int_{\p B_1} w_r^2d\sigma}
		&= \frac12\frac{\int_{\p B_2} w_{r}^2d\sigma}{\int_{\p B_1} w_r^2d\sigma}\\
		&= \frac{1}{2} \int_{\p B_2} \tilde{w}_r^2d\sigma \\
		&\to \frac12 \int_{\p B_2} v^2d\sigma\\
		& = \frac12 \frac{\int_{\p B_2} (2)^3\sin^2(\frac32\theta)d\sigma}{\int_{\p B_1} (1)^3\sin^2(\frac32\theta)d\sigma}\\
		&= 8.
	\end{align*}
	Then denoting $f(r) := \norm{w_r}_{L^2(\p B_1)}$ we have that for all $\lambda >0$ there exists an $r(\lambda) < + \infty$ so that 
	\begin{equation*}
		f(2r) \leq 2^{\frac{3}{2}+\lambda}f(r), \hspace{2mm} \text{for all } r\geq r(\lambda).
	\end{equation*}
	Iterating this we find that 
	\begin{equation*}
		f(r)\leq C(\lambda) r^{\frac32+\lambda} \hspace{2mm} \text{for all } r\geq r(\lambda).
	\end{equation*}
\end{proof}

We can now show the following almost optimal growth estimate on the coincidence set. 

\begin{proposition}\label{Cdelta}
	Given $\delta \in (0,1)$ there exists an $r_2(\delta) > 1$ so that 
	\begin{equation}\label{deltagrowth}
		C \cap \{y_2 \geq r_2(\delta)\} \subset \{\abs{y_1}^2 \leq y_2^{1+\delta}\}
	\end{equation}
	and
	\begin{equation}\label{coincidencesetbounded}
		C \cap \{y_2 \leq r_2(\delta)\} \subset \{-r_2(\delta)^{\frac{1+\delta}{2}} \leq y_1 \leq r_2(\delta)^{\frac{1+\delta}{2}} \} \cap \{0 \leq y_2 \leq r_2(\delta) \}
	\end{equation}
\end{proposition}

\begin{proof}
	We first note that \eqref{coincidencesetbounded} follows immediately from \eqref{deltagrowth} and the fact that $\p_2 u \leq 0$. 
	We will now prove \eqref{deltagrowth} by contradiction. If \eqref{deltagrowth} did not hold then there exists some $\delta \in (0,1)$ such that for every $k \in \N$ there exists some point $x^k \in C$ such that $x^k_2 \geq k$ and $\abs{x_1^k} \geq \left( x_2^k \right)^{\frac{1+\delta}{2}}$. With no loss of generality we can take a subsequence, that we do not relabel, such that $x_1^k \geq \left( x_2^k \right)^{\frac{1+\delta}{2}}$.
	
	We define $\rho_k = x_2^k$  and observe that since $w(x) = -\frac{1}{2}x_1^2$ on $C$ we have that
	\begin{equation*}
		w(x^k) = -\frac{1}{2} \left(x^k_1 \right)^2 \leq -\frac12 \rho_k^{1+\delta}
	\end{equation*}
	and
	\begin{equation*}
		\p_1 w(x^k) = - x^k_1 \leq -\rho_k^{\frac{1+\delta}{2}}.
	\end{equation*}
	Now since $\p_{11} w \leq 0$ (c.f. Remark \ref{signedremark}) we have using the above estimates, as well as \eqref{3/2growth} for $k$ large enough, that
	\begin{equation*}
		\abs{w(\rho_k,\rho_k)} \geq c(\delta) \rho_k^{\frac{3+\delta}{2}}. 
	\end{equation*}
	Now using \eqref{doubling} with $\lambda = \frac{\delta}{4}$ we have for $k$ large enough that
	\begin{equation*}
		\abs{\tilde{w}(\rho_k,\rho_k)}= \frac{\abs{w(\rho_k,\rho_k)}}{\norm{w_{\rho_k}}_{L^2(\p B_1)}} \geq c(\delta) \rho_k ^{\frac{\delta}{4}},
	\end{equation*}
	which contradicts the pointwise convergence of $\abs{\tilde{w}(\rho_k,\rho_k)} \to \abs{v(1,1)} < +\infty$.
\end{proof}

\section{Sharp growth estimates of $u-p$}\label{section:sharpgrowth}
We will define 
\begin{equation*}
	H(r,w) = r^{1-n}\int_{\p B_r} w^2d\sigma = \int_{\p B_1} w_r^2d\sigma,
\end{equation*}
and
\begin{equation*}
	D(r,w) = r^{2-n} \int_{B_r} \abs{\nabla w}^2 = \int_{B_1} \abs{\nabla w_r}^2,
\end{equation*}
so that the Almgren frequency is 
\begin{equation*}
	\phi(r,w) = \frac{D(r,w)}{H(r,w)}.
\end{equation*}

We will now collect some known differential identities on the quantities $H(r,w)$ and $\phi(r,w)$. 

\begin{lemma}
	\label{dHlemama}
	For each $r>0$ we have that
	\begin{equation}\label{integrationparts}
		\int_{\p B_1} w_r(w_r)_{\nu}d\sigma = \int_{B_1} \abs{\nabla w_r}^2 + \int_{B_1} w_r \Delta w_r.
	\end{equation}
	Consequently,
	\begin{equation}\label{Hprimepos}
		\frac{d}{dr}H(r,w) = \frac{2}{r} \int_{\p B_1} w_r(w_r)_{\nu}d\sigma \geq 0
	\end{equation}
	and
	\begin{equation}\label{Hprime}
		\frac{d}{dr}\left(\log(H(r,w))\right) = \frac{2}{r} \left(\phi(r,w) + \frac{\int_{B_1} w_r\Delta w_r}{H(r,w)}\right).
	\end{equation}
\end{lemma}

\begin{proof}
	Integrating by parts the quantity $\int_{B_1} \abs{\nabla w_r}^2$ we immediately obtain \eqref{integrationparts}. A direct computation yields the identity in \eqref{Hprimepos} while \eqref{integrationparts} together with Remark \ref{signedremark} establishes the non-negativity of $\frac{d}{dr}H(r,w)$. Finally, \eqref{Hprime} follows from both the identity in \eqref{Hprimepos} and \eqref{integrationparts} after dividing through by $H(r,w)$.
\end{proof}

\begin{lemma}
	\label{AHlemama}
	For each $r>0$ we have 
	\begin{equation}\label{frequencydiff}
		\frac{d}{dr}(\phi(r,w)) \geq -2\frac{r^3}{H(r,w)}\int_{B_1\cap \{u_r =0\}}x_1^2.
	\end{equation}
\end{lemma}

\begin{proof}
	The proof is also a direct computation. Using the notation $\p_r=\frac{d}{dr}$,  we first observe that
	\begin{align*}
		\frac{1}{2}H(r,w)^2 \p_r\left(\frac{D(r,w)}{H(r,w)}\right)
		&=  \int_{B_1} \nabla w_r \cdot \nabla (\p_r w_r) \int_{\p B_1} w_r^2d\sigma -  \int_{B_1}\abs{ \nabla w_r}^2 \int_{\p B_1} w_r\p_rw_r d\sigma\\
		&= \int_{\p B_1} w_r^2d\sigma \left\{-\int_{B_1}\p_rw_r\Delta w_r + \int_{\p B_1}\p_rw_r (x\cdot \nabla w_r)d\sigma\right\} \\
		&\hspace{12mm}- \int_{\p B_1} w_r \p_rw_rd\sigma \left\{- \int_{B_1}w_r\Delta w_r + \int_{\p B_1} w_r (x\cdot \nabla w_r)d\sigma\right\}.
	\end{align*}
	Now since $x \cdot \nabla w_r = r\p_r w_r$ we have that
	\begin{align*}
		\frac{1}{2}H(r,w)^2 \p_r\left(\frac{D(r,w)}{H(r,w)}\right)
		&=  r\left\{\int_{\p B_1} w_r^2d\sigma\int_{\p B_1}(\p_rw_r)^2d\sigma  - \left( \int_{\p B_1} w_r \p_rw_r d\sigma\right)^2 \right\} \\
		&\hspace{12mm} + \int_{B_1}w_r\Delta w_r \int_{\p B_1} w_r \p_rw_rd\sigma - \int_{\p B_1} w_r^2d\sigma \int_{B_1} \p_r w_r \Delta w_r.
	\end{align*}
	The bracket on the right is non-negative due to the Cauchy-Schwarz inequality while the second term is non-negative thanks to \eqref{Hprimepos} and Remark \ref{signedremark}. Finally, since $\Delta w_r =- r^2\chi_{\{u_r=0\}}$,  we have that $\p_rw_r\Delta w_r = r^3x_1^2 \chi_{\{u_r =0\}}$. Thanks to these considerations we find that
	\begin{align*}
	\frac{1}{2}H(r,w)^2 \p_r\left(\frac{D(r,w)}{H(r,w)}\right)
		&\geq  -  r^3H(r,w) \int_{B_1\cap \{u_r =0\}}x_1^2
	\end{align*}
	which establishes \ref{frequencydiff}.
\end{proof}

We can now give the almost monotonicity of the Almgren frequency. We recall that given $\eps >0 $ and $\delta > 0$ the quantities $r_1 (\eps)$ and $r_2 (\delta)$ are obtained from Proposition \ref{frequencyconvergesprop} and Proposition \ref{Cdelta} respectively. From this point onwards we will use the fact that $n=2$.

\begin{lemma}\label{almostmono}
	Let $\eps = \delta = \frac{1}{16}$ and define $r_0=\max\{r_1(\eps),r_2(\delta)\}$. There exists a constant $C=C(r_0)$ so that
	\begin{equation*}
		\frac{d}{dr}(\phi(r,w) - C r^{-\frac{1}{4}})\geq 0
	\end{equation*}
	for $r\geq r_0$. 
\end{lemma}

\begin{proof} The result will follow from Lemma \ref{AHlemama} after estimating $H(r,w)$ from below and $\int_{B_1\cap \{u_r =0\}}x_1^2$ from above. \\
	\textit{Step 1: Control of $H(r,w)$ from below:}\\
	Since $w\Delta w \geq 0$ (c.f. Remark \ref{signedremark}) we find using \eqref{Hprime} and \eqref{frequencyconverges} that
	\begin{equation*}
		\frac{d}{dr}(\log(H(r,w)) \geq \frac{2}{r} \left( \frac{3}{2} - \eps \right)
	\end{equation*}
	for $r \geq r\naught$. Integrating this identity from $r\naught$ to $r$ we find
	\begin{equation*}
		\int_{r\naught}^r 	\frac{d}{d\rho}(\log(H(\rho,w))) d\rho \geq \int_{r\naught}^r \frac{2}{\rho}\left(\frac{3}{2} - \eps \right) d\rho
	\end{equation*}
	so that
	\begin{equation}\label{Hlowerbound}
		H(r,w) \geq C r^{3-2\eps}.
	\end{equation}
	\textit{Step 2: Controlling $\int_{B_1\cap \{u_r =0\}}x_1^2$ from above:}\\
	By Proposition \ref{Cdelta} we have
	\begin{align}
		\int_{B_1\cap \{u_r =0\}}x_1^2
		&\leq  \int_{ \{-r\naught^{\frac{1+\delta}{2}}r^{-1} \leq x_1 \leq r\naught^{\frac{1+\delta}{2}}r^{-1} \} \cap \{0 \leq x_2 \leq \frac{r\naught}{r}\}} x_1^2 dx \nonumber \\
		&\hspace{25mm}+\int_{B_1 \cap \{\abs{x_1}^2 \leq r^{\delta-1}x_2^{1+\delta}\} \cap \{\frac{r\naught}{r} \leq x_2 \leq 1\}} x_1^2 dx \nonumber \\
		&\leq Cr^{-4} + \int_{\frac{r\naught}{r}}^1 \int_{- r^{\frac{\delta-1}{2}}x_2^{\frac{1+\delta}{2}}}^{ r^{\frac{\delta-1}{2}}x_2^{\frac{1+\delta}{2}}} x_1^2 dx_1dx_2\nonumber \\
		&\leq C r^{\frac32 (\delta -1)} \nonumber.
	\end{align}
	\textit{Step 3: Controlling $\phi(r,w)$ from below:}\\
	By \eqref{frequencydiff}, Step 1, and Step 2, we have that
	\begin{align*}
		\frac{d}{dr}(\phi(r,w))
		&\geq -C r^{\frac32 (\delta -1) + 2\eps},
	\end{align*}
	and since $\eps = \delta = \frac1{16}$ we have that $\frac32\delta + 2\eps \leq \frac{1}{4}$ which concludes the proof. 
\end{proof}

The main result of this section is the following sharp growth estimate on $w=u-p$. 
\begin{proposition}\label{sharpgrowthestimates}
	Let $\eps, \delta$ and $r\naught$ be as in Lemma \ref{almostmono}. There exists a constant $C=C(r_0) $ such that
	\begin{equation}
		H(r,w) \leq C r^{3}
	\end{equation}
	for $r\geq r_0$. 
\end{proposition}

\begin{proof}
	We first observe that using \eqref{Hlowerbound} we obtain (as in Step 2 of the previous proof) that
	\begin{align*}
		\frac1r\frac{\int_{B_1} w_r\Delta w_r}{H(r,w)}
		\leq C r^{2\eps} \int_{B_1 \cap \{u_r =0\}} x_1^2 
		\leq C r^{-\frac54}
	\end{align*}
	for $r \geq r_0$. Using now \eqref{Hprime} and Lemma \ref{almostmono} we find 
	\begin{align*}
		\frac{d}{dr}(\log(H(r,w))) 
		&\leq \frac{2}{r} \left(\phi(r,w) - Cr^{-\frac{1}{4}} \right) + Cr^{-\frac{5}{4}} \\
		&\leq \frac{3}{r} + Cr^{-\frac{5}{4}}
	\end{align*}
	where in the last step we used that $\lim_{r\to\infty} \phi(r,w) = \frac{3}{2}$. Integrating this we obtain the result. 
\end{proof}

We have the following immediate consequence. 
\begin{corollary}\label{sharpcor}
	There exists a constant $C=C(r_0)$ such that for $r \geq r_0$ we have
	\begin{equation}\label{L2growth}
		\left(\strokedint_{B_r} w^2\right)^{\frac{1}{2}} \leq C r^{\frac{3}{2}}.
	\end{equation}
\end{corollary}
\begin{proof}
	We observe by integrating that
	\begin{align*}
		\strokedint_{B_r} w^2 
		&= \frac{1}{\abs{B_r}} \int_0^r \int_{\p B_{\rho}} w^2 d\sigma d\rho \\
		&= \frac{1}{\abs{B_r}} \int_0^{r_0} \int_{\p B_{\rho}} w^2 d\sigma d\rho + \frac{1}{\abs{B_r}} \int_{r_0}^r \int_{\p B_{\rho}} w^2 d\sigma d\rho \\ 
		&\leq \frac{1}{\abs{B_r}} \int_{B_{r_0}}w^2  + \frac{C}{\abs{B_r}} \int_{r_0}^r \rho^{n-1+3} d\rho \\ 
		&\leq Cr^{3}
	\end{align*}
	which establishes \eqref{L2growth}.
\end{proof}

\section{Refined blow-down and matching}\label{section:refined:blowdown}
In this section we use the sharp growth estimates obtained for a solution $u$ of \eqref{obstacleproblem} to show that the blow-down sequence, $r^{-\frac{3}{2}}(u-p)(r\cdot)$, has a limit as $ r\to \infty$. We will then find a paraboloid solution such that its blow-down limit coincides with that of $u$. As before we set $w= u-p$ and we recall that $\hat{v}_{3/2}$ is defined in \eqref{wsin} and is the $\frac{3}{2}$-homogeneous solution to the thin obstacle problem 

\begin{proposition}\label{refinedblowdown}
	Given a sequence $(r_k)_{k\in \N}$ such that $r_k \to \infty$ there exists a subsequence and some constant (that may depend on the sequence) $\alpha_u > 0$ such that 
	\begin{equation}\label{perciseblowdown}
		\frac{w(r_kx)}{r_k^{\frac{3}{2}}} \to \alpha_u v \text{\hspace{1mm} in $W^{1,2}_{\loc}\cap C^0_{\loc}(\R^2)$. }
	\end{equation}
\end{proposition}

\begin{proof}
	The proof of \eqref{perciseblowdown} follows the same lines as that of Proposition \ref{blowdownanalysis}, however, in this case the uniform $W^{1,2}$ bounds follow from the sharp growth estimates on $u-p$. Precisely, since $p$ is itself a global solution to the obstacle problem we can apply Lemma \ref{differenceenergyinequality} and then \eqref{L2growth} to find for each $R>1$ and each $r>0 $ that, setting $\hat{w}_r(x):= r^{-3/2}w(rx)$,
	\begin{align*}
		 \int_{B_R} \abs{\nabla \hat{w}_r}^2 \leq \frac{C}{R^2}\int_{B_{2R}} \left(r^{-3/2}w(r\cdot)\right)^2 =  \frac{C}{r^3R^2}\int_{B_{2Rr}} w^2 \leq C R.
	\end{align*}
	Moreover, from \eqref{L2growth} we have for each $R>1$ that
	\begin{equation}\label{wr:refined:L2bound}
		\int_{B_R} \hat{w}_r^2 \leq CR^3 \text{\hspace{2mm} for all } r > 0,
	\end{equation}
	so that 
	\begin{equation*}
		\norm{\hat{w}_r}_{W^{1,2}(B_R)} \leq C(R) \text{\hspace{2mm} for all } r > 0. 
	\end{equation*}
	Arguing then as in the proof of Proposition \ref{blowdownanalysis} we find, up to a subsequence, that $\hat{w}_{r_k} \to v_0$ in $W^{1,2}_{\loc}\cap C^0_{\loc}(\R^2)$, where $v_0$ is a global solution to the thin obstacle problem \eqref{thinobstacleglobal}, that may depend on the sequence chosen. However, since $\lim_{r\to\infty} \phi(r,w) = \frac{3}{2}$ by Proposition \ref{frequencyconvergesprop}, we can argue as in the derivation of Corollary \ref{corollary:blowdown:homogeneous} to find that $\hat{w}_{r_k} \to \alpha_u \hat{v}_{3/2}$ in $W^{1,2}_{\loc}\cap C^0_{\loc}(\R^2)$ where $\alpha_u := \lim_{k\to\infty} \norm{\hat{w}_{r_k}}_{L^2(\p B_1)}$.
\end{proof}

In order to complete the matching we must show uniqueness of the blow-down limits obtained in Proposition \ref{refinedblowdown}. This will be achieved using the following monotonicity formula which appears in \cite[Lemma 5.3 and Lemma B.1]{figalli2020generic} for odd integer homogeneous solutions to the thin obstacle problem. 

\begin{lemma}\label{monotonicitysigniorini}
	For each $r>0$ we have that
	\begin{equation*}
		\frac{d}{dr} \left(\frac{1}{r^{\frac{3}{2}}} \int_{\p B_1}  \hat{v}_{3/2} w_r d\sigma\right) \geq 0.
	\end{equation*}
\end{lemma}
\begin{proof}
	We first observe that by integrating by parts twice $\int_{B_1} w_r \Delta v$, we have
	\begin{equation*}
		\int_{\p B_1} \hat{v}_{3/2}\p_{\nu} w_r d\sigma = - \int_{B_1} w_r \Delta \hat{v}_{3/2} + \int_{\p B_1}w_r \p_{\nu} \hat{v}_{3/2} d\sigma + \int_{B_1} \hat{v}_{3/2}\Delta w_r.
	\end{equation*}
	Now since $\hat{v}_{3/2}$ is $\frac32$-homogeneous we have that $\p_{\nu} \hat{v}_{3/2} = \frac{3}{2}\hat{v}_{3/2}$ on $\p B_1$. Moreover, since $\hat{v}_{3/2} \leq 0$ for $x_2\geq 0$ we have that $\hat{v}_{3/2}\chi_{\{u_r=0\}}\leq 0$ for every $r>0$. Finally, $w_r\Delta \hat{v}_{3/2} = 0$ since $w_r = 0$ on $\supp(\Delta \hat{v}_{3/2}) = \{x_1=0, x_2 \geq 0\}$. Altogether we obtain
	\begin{align*}
		\frac{d}{dr} \left(\frac{1}{r^{\frac{3}{2}}} \int_{\p B_1}  \hat{v}_{3/2} w_rd\sigma \right) 
		&= -\frac{3}{2}r^{-\frac52}\int_{\p B_1} \hat{v}_{3/2}w_r d\sigma\\& \quad +r^{-\frac{5}{2}} \left(\frac32\int_{\p B_1} \hat{v}_{3/2}w_r d\sigma- \int_{ B_1}w_r\Delta \hat{v}_{3/2} + \int_{B_1} \hat{v}_{3/2}\Delta w_r\right) \\
		&= -r^{-\frac12}  \int_{B_1} \hat{v}_{3/2}\chi_{\{u_r=0\}} \geq 0.
	\end{align*}
\end{proof}

\begin{proposition}\label{uniqueness}
	There exists a unique constant $\alpha_u$ such that
	\begin{equation*}
		\lim_{r\to\infty} \frac{w(rx)}{r^{\frac{3}{2}}} \to \alpha_u \hat{v}_{3/2} \text{\hspace{2mm} in } W^{1,2}_{\loc} \cap C^0_{\loc}(\R^2).
	\end{equation*}
\end{proposition}
\begin{proof}
	Suppose there exists two different accumulation points along sequences $r^{(i)}_k$ for $i=1,2$. That is, there exists constants $\alpha^{(i)}$ such that
	\begin{equation*}
		\alpha^{(i)} \hat{v}_{3/2} = \lim_{k\to \infty} \frac{w(r^{(i)}_kx)}{(r^{(i)}_k)^{\frac{3}{2}}}
	\end{equation*} 
	for $i=1,2$. Supposing that $\alpha^{(1)}_u > \alpha^{(2)}_u$ we have that $r\mapsto r^{-\frac32}\int_{\p B_1} w_{r}(\alpha^{(1)}_u-\alpha^{(2)}_u)\hat{v}_{3/2}$ is monotone non-decreasing by Lemma \ref{monotonicitysigniorini}. Moreover by H\"older's inequality and \eqref{wr:refined:L2bound} it is bounded from above and hence it is has a unique limit as $r \to \infty$. This implies that
	\begin{equation*}
		\int_{\p B_1} \alpha^{(1)}_u(\alpha^{(1)}_u-\alpha^{(2)}_u)\hat{v}_{3/2}^2d\sigma =  \int_{\p B_1} \alpha^{(2)}_u(\alpha^{(1)}_u-\alpha^{(2)}_u)\hat{v}_{3/2}^2d\sigma.
	\end{equation*}
	Re-arranging this we obtain
	\begin{equation*}
		\int_{\p B_1} (\alpha^{(1)}_u-\alpha^{(2)}_u)^2\hat{v}_{3/2}^2d\sigma=0
	\end{equation*}
	and so $\alpha_u^{(1)} = \alpha_u^{(2)}$.
\end{proof}

When $u_{\gamma P}$ is a paraboloid solution as constructed in Proposition \ref{paraboloidsolutions} we will denote the blow-down limit as $\alpha_{\gamma}v$. As a consequence of the uniqueness of the blow-down limit, we have the following identity for $\alpha_{\gamma}$. 

\begin{corollary}\label{scalingparab}
	Let $u_{\gamma P}$ be a paraboloid solution and $\alpha_{\gamma} v$ its blow-down. Then  
	\begin{equation}\label{scalingalpha}\nonumber
		\alpha_{\gamma} = \gamma^{\frac{1}{2}} \alpha_1.
	\end{equation}
\end{corollary}
\begin{proof}
	We use Proposition \ref{expansionthm} and Lemma \ref{scaling} to see that
	\begin{equation*}
		r^{-\frac32}(u_{\gamma P} - p)(rx) = r^{-\frac32} V_{\gamma P}(rx) = \gamma^2 r^{-\frac32}V_{P}(\frac{1}{\gamma}rx).
	\end{equation*}
	Using Proposition \ref{propuniqueness} we pass to the limit as $r\to{\infty}$ in the above and find that
	\begin{equation*}
		\alpha_{\gamma} \hat{v}_{3/2}(x) = \gamma^2 \alpha_1 \hat{v}_{3/2}(\frac{1}{\gamma}x) = \gamma^{\frac{1}{2}}\alpha_1 \hat{v}_{3/2}(x),
	\end{equation*}
	where in the last step we used the fact that $\hat{v}_{3/2}$ is $3/2$-homogeneous. 
\end{proof}

We can now match the second order blow-down with translations of paraboloid solutions. To this end, given $\sigma \in \R$ denote by $u_{\sigma}$ the global solution that has $\gamma P -\sigma e_1$ as its coincidence set. Observe that $u_{\sigma}(x_1,x_2) = u_{\gamma P}(x_1+\sigma,x_2)$.

\begin{proposition}\label{matching}
	Given a $x_2$-monotone solution $u$ to the obstacle problem as in Definition \ref{x2mondef} there exists a paraboloid solution $u_{\gamma P}$ such that
	\begin{equation}\label{matched:notranslation}
		\lim_{r\to\infty}\frac{(u-u_{\gamma P})(r\cdot)}{r^{\frac{3}{2}}} = 0 \text{ in } W^{1,2}_{\loc}\cap C^0_{\loc}(\R^2)
	\end{equation}
	Moreover given $\sigma \in [-1,1]$ we have that
	\begin{equation}\label{matched:translation}
		\lim_{r\to\infty}\frac{(u-u_{\sigma})(r\cdot)}{r^{\frac{3}{2}}} = 0 \text{ in } C^0_{\loc}(\R^2) .
	\end{equation}
\end{proposition}

\begin{proof}
	Given $\alpha_u$ from Proposition \ref{uniqueness} let $\gamma = \left(\frac{\alpha_u}{\alpha_1}\right)^2$  so that $\alpha_{\gamma} = \alpha_u$ (by Corollary \ref{scalingparab}) and hence \eqref{matched:notranslation} follows directly from Proposition \ref{refinedblowdown}. \\
	Now to obtain \eqref{matched:translation} we use a Taylor expansion with remainder along with the $C^{1,1}$ regularity of the solution $u_{\gamma P}$ to see that for each $R>0$,
	\begin{equation*}
		\sup_{B_R}\frac{\abs{u_{\gamma P}(rx_1+\sigma, rx_2)-u_{\gamma P}(rx_1,rx_2)}}{r^{\frac{3}{2}}} \leq C(R)r^{-\frac{1}{2}}.
	\end{equation*}
	This along with \eqref{matched:notranslation} establishes the $C^0_{\loc}$ convergence. 
\end{proof}

\section{Proof of Theorem \ref{maintheorem}}\label{section:conclusion}

We begin this section with the following Proposition. 

\begin{proposition}\label{propuniqueness}
	Let $u_1, u_2$ be two $x_2$-monotone solutions. If there exists some non-empty open set $\Omega \subset \{u_1>0\} \cap \{u_2 >0\}$ such that $(u_1-u_2)\vert_{\Omega} \equiv 0$ then $u_1 \equiv u_2$. 
\end{proposition}

\begin{proof}
	We first observe that the set $\Sigma := \{u_1>0\}\cap \{u_2>0\}$ is connected 
since it is the subgraph of a function of the $x_1$-variable.
Since $\Delta (u_1-u_2) \equiv 0$ in the connected set $\Sigma$, and $(u_1-u_2) \equiv 0$ on the open set $\Omega$, we must have that $u_1 \equiv u_2$ in $\Sigma$. It follows that $u_1 = u_2 =0$ on $\p \Sigma$ which completes the proof.
\end{proof}
The following is a useful consequence. 
\begin{corollary}\label{corollary:around:the:tip}
	Let $u_1,u_2$ be two $x_2$-monotone solutions. If there exists $r >0$ such that $\{u_1 =0\} = \{u_2 =0\}$ in $B_r^{+}$ or in $B_r^{-}$ , then $u_1 \equiv u_2$.
\end{corollary}
\begin{proof}
	If the coincidence sets of $u_1$ and $u_2$ coincide in $B_r^+$ then there exists a ball $B$ such that $\left(B \cap \left(\{u_1=0\}\cap \{u_2 =0\}\right)\right)^{\circ} \neq \emptyset$ and $\left(B\cap \left(\{u_1 > 0\} \cap \{u_2 >0\}\right)\right)^{\circ} \neq \emptyset$ (see Figure \ref{prop:tip:pic}). In $B$ we have that $\Delta (u_1-u_2) = 0$ and so we can conclude that $u_1 -u _2 \equiv 0$ in $B$. On the other hand, $\left(B\cap \left(\{u_1 > 0\} \cap \{u_2 >0\}\right)\right)^{\circ} \neq \emptyset$ and so there exists some open set $\Omega \subset B\cap \left(\{u_1 > 0\} \cap \{u_2 >0\}\right)$ such that $u_1 -u_2 \equiv 0$ on $\Omega$. Applying Proposition \ref{propuniqueness} we conclude that $u_1 \equiv u_2$.
\end{proof}
Finally, we will require the following comparison result for any two solutions of the obstacle problem.
\begin{proposition}\label{proposition:unbounded}
	Suppose that $u_1, u_2$ are two solutions of the obstacle problem. Then each connected component of the set $\{u_1 > u_2\}$ must be unbounded. 
\end{proposition}
\begin{proof}
	Suppose that there is a connected component $\Omega$ of $\{u_1 > u_2\}$ that is bounded. Then $\abs{u_1-u_2} = 0$ on $\p \Omega$ but it is subharmonic by Lemma \ref{differenceenergyinequality}. Hence, $u_1\equiv u_2$ in $\Omega$, a contradiction. 
\end{proof}
\begin{figure}
	\centering
	\includegraphics[scale=1.5]{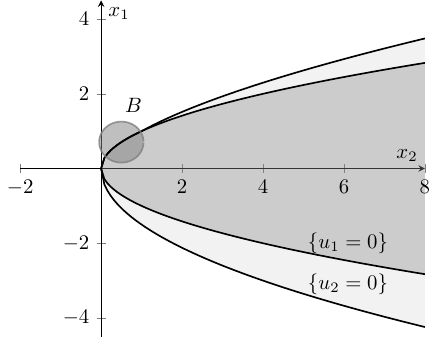}
	\caption{Setting of the proof of Corollary \ref{corollary:around:the:tip}}
	\label{prop:tip:pic}
\end{figure}

We can now give the 
\begin{proof}[Proof of Theorem \ref{maintheorem}:]
	Given any global solution $u$ as in Definition \ref{x2mondef} we let $u_{\gamma P}$ be the paraboloid solution given by Proposition \ref{matching}. Moreover we recall the notation $u_{\sigma} = u_{\gamma P}(x_1+\sigma,x_2)$ for $\sigma \in \R$ and we will denote the coincidence set of $u_{\sigma}$ as $P_{\sigma}$. \\
	We will show that $u=u_{\gamma P}$ by a contradiction argument. If $u$ and $u_{\gamma P}$ were not identical, we would have that $C$ and $\gamma P$ touch at the origin but cannot be identical in any half ball around the origin by Corollary \ref{corollary:around:the:tip}. Since $\p C$ is analytic, there exists some small $\rho >0$ such that the free boundaries of $u$ and $u_{\gamma P}$ do not intersect in $B_{\rho}$ except at the origin (recall that $\p \{u_{\gamma P} >0\})$ is a paraboloid).
	
	\begin{claim}{1}
	 There exists some $\eps_0 > 0$ such that the boundary of $P_{\sigma}$ intersects $\p C$ at least twice in $B_{\rho}$ for $0 < \sigma \leq \eps_0$ or $-\eps_0 \leq \sigma < 0$ . 
	\end{claim}
	\begin{proof}[Proof of Claim 1:]
		We define
		\begin{equation*}
			\eps_0 = \frac{1}{2}\min\{ \dist(\gamma P \cap B_{\rho}^+, C \cap B_{\rho}^+), \dist(\gamma P \cap B_{\rho}^-, C \cap B_{\rho}^-)\}.
		\end{equation*} 
		Now, there are three possible cases, either
		\begin{enumerate}
			\item $C \cap B_{\rho} \subset \gamma P$, or,
			\item $\gamma P \cap B_{\rho} \subset C$ , or,
			\item $\gamma P \cap B^+_{\rho} \subset C$ while $C \cap B_{\rho}^- \subset \gamma P$ (or vice versa).
		\end{enumerate}
		We examine the first case. Since $\p \{u_{\gamma P} > 0 \}$ and $\p C$ touch at the origin, passing from $u_{\gamma P}$ to $u_{\sigma}$ for $0 \neq \abs{\sigma} \leq \eps_0$ will produce at least two points of intersection (see Figure \ref{slidepic}). Case 2 is handled in an identical manner. 
		
		In the third case, passing from $u_{\gamma P}$ to $u_{\sigma}$ for $-\eps_0 \leq \sigma < 0$ will produce at least three points of intersection. Indeed, since $C$ and $\gamma P$ touch at the origin there will be at least two points of intersection between $C$ and $P_{\sigma}$ in $B_{\rho}^+$ and at least one in $B_{\rho}^-$. In the vice versa case one would take $0 < \sigma \leq \eps_0$, which concludes the proof of the claim. 
	\end{proof}
	
	Now for each $\sigma$ given by Claim 1, there must exist at least two curves $L^{\sigma}_1, L^{\sigma}_2 \subset \R^2$ along which $u-u_{\sigma}=0$ and $u-u_{\sigma}$ changes sign across them (see Figure \ref{slidepic}). Moreover, by Proposition \ref{proposition:unbounded} the curves $L_1^{\sigma}$ and $L_2^{\sigma}$ must be unbounded, and so there are three unbounded sets $\Omega^{\sigma}_i$ for $i=1,2,3$ such that $u-u_{\sigma}=0$ on $\p\Omega^{\sigma}_i$ (see Figure \ref{slidepic}). For $i=1,2$, $\Omega_i^{\sigma}$ will have boundary given by $L_i^{\sigma}$ and the unbounded component of $\p C \cup \p P_{\sigma}$ that intersects $L^{\sigma}_i$ only and along which $u-u_{\sigma}=0$.  The set $\Omega_3^{\sigma}$ will have boundary given by $L_1^{\sigma}$, $L^{\sigma}_2$ and the component of  $\p C \cup \p P_{\sigma}$ along which $u-u_{\sigma}=0$ and that lies between the points of intersection with $L^{\sigma}_1$ and $L^{\sigma}_2$.
	
	\begin{figure}
		\centering
		\includegraphics[scale=1.5]{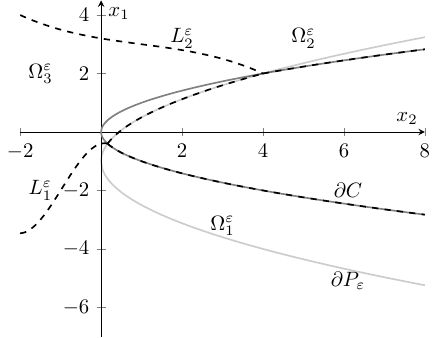}
		\caption{The sliding of $\gamma P$ downwards produces at least three unbounded regions. The dashed lines represent their boundaries.}
		\label{slidepic}
	\end{figure}
	
	Observe that for $i=1,2,3$ we have that $(\Omega_i^{\sigma} \cap \left(\{u>0\} \cap \{u_{\sigma} > 0\}\right))^{\circ} \neq \emptyset$ and $\p \Omega_i^{\sigma} \cap \p B_r \neq \emptyset$ for every $r \geq \rho$. 
	
	Since $\abs{u-u_{\sigma}}$ is a non-negative subharmonic function by Lemma \ref{differenceenergyinequality}, we can apply Theorem \ref{prop:modACF} to the functions
	 $v^{\sigma}_i= \abs{u-u_{\sigma}}\chi_{\Omega^{\sigma}_i}$
	to see that the functional 
	\begin{equation*}
		\Phi(r;\sigma) := \frac{1}{r^9} \prod_{i=1}^3 \int_{B_r} \abs{\nabla v^{\sigma}_i}^2 
	\end{equation*}
	is monotone non-decreasing in $r$ for $r\geq\rho$. We then find using Lemma \ref{differenceenergyinequality} again that
	\begin{equation*}
		\Phi(r;\sigma) \leq C  \prod_{i=1}^3 \int_{B_{2}} (r^{-3/2}(v_i^{\sigma})_r)^2.
	\end{equation*}
	By Proposition \ref{matching} we have that the right hand side above converges to $0$ as $r\to \infty$. Since $\Phi(r;\sigma)$ is non-negative and monotone non-decreasing for $r \geq \rho$ we have that $\Phi(r;\sigma) = 0$ for all $r \geq \rho$. Since $\Phi(\rho;\sigma) =0$ implies that $\Phi(r;\sigma)=0$ for all $r\leq \rho$ we have that
	\begin{equation*}
		\Phi(r;\sigma) = 0 \text{ for all } r\geq 0 .
	\end{equation*}
	 So for each $\sigma$ given by Claim 1, one of the $v_i^{\sigma} \equiv 0$ which means there exists some ball $B^{\sigma} \subset \{u >0\} \cap \{u_{\sigma}>0\}$ such that $(u-u_{\sigma})\vert_{B^{\sigma}} \equiv 0$. Proposition \ref{propuniqueness} then implies that $u \equiv u_{\sigma}$. This is clearly not possible for each $\sigma$ as in Claim 1, and so we must have instead that $C = \gamma P$.
\end{proof}

\begin{appendices}
\section{Proof of Lemma \ref{holderboundlemma}}\label{holderboundapp}
In this appendix we provide the proof of Lemma \ref{holderboundlemma} that is in \cite[Lemma A5]{franceschini2021c}. We note that a Lipschitz estimate can be obtained in our setting (see for instance \cite{figalli2019fine}), however we will not do that here. For simplicity we will work in $B_8$ and given $w: B_2 \to \R$ Lipschitz we introduce the notation 
\begin{equation*}
	[\delta_j w]_{\sigma} = \sup_{x \in B_{1}, \abs{t} \leq 1} \frac{\abs{u(x + te_j)-u(x)}}{t^{\sigma}}
\end{equation*}
for $j\in \{1,\dots,n\}$ and $\sigma \in (0,1]$. 
We first note the following estimate on the H\"older semi-norms of $w \in Lip(B_2)$ (see \cite[Lemma C.1]{franceschini2021c}).
\begin{lemma}
	\label{lemma:holder:characterisation}
	Let $w \in Lip (B_2)$. For each $p>1$ and each $\sigma \in (0,1]$ there exists a constant $C=C(n, p,\sigma) > 0$ such that
	\begin{equation*}
		[w]_{C^{\beta}(B_{1})} \leq C \left(\sum_{j=1}^{n-1}[\delta_jw]_{\sigma} + \norm{\p_n w}_{L^p(B_2)}\right)
	\end{equation*}
	where $\beta = \frac{\sigma(p-1)}{\sigma p +n -1}$.
\end{lemma}
We can now give the
\begin{proof}[Proof of Lemma \ref{holderboundlemma}]
	The aim is to apply Lemma \ref{lemma:holder:characterisation} when $n=2$ to the difference $w_r=(u-p)_r$ with $p=2$ and $\sigma =1$.
	We first note that as a consequence of Lemma \ref{differenceenergyinequality} there exists a universal constant $C$ such that
	\begin{equation}
		\label{equation:Linfty:L2:bound}
		\norm{w_r}_{L^{\infty}(B_{4})} \leq C \norm{w_r}_{L^2(B_8)}.
	\end{equation}
	and
	\begin{equation}
		\label{equation:caccioppolli}
		\norm{\nabla w_r}_{L^{2}(B_{4})} \leq C \norm{w_r}_{L^2(B_8)}.
	\end{equation}
	Indeed, $\abs{w_r}$ is subharmonic, and so for each $x \in B_4$ we have that
	\begin{align*}
		\abs{w_r}(x) \leq \strokedint_{B_4(x)} \abs{w_r} \leq C\norm{w_r}_{L^2(B_8)},
	\end{align*}
	which establishes \eqref{equation:Linfty:L2:bound}. Moreover, \eqref{equation:caccioppolli} follows directly from Lemma \ref{differenceenergyinequality}.
	We will now show that there exists a universal constant $C$ such that
	\begin{equation}
		\label{equation:holder:quotient}
		[\delta_2w_r]_1 \leq C \norm{w_r}_{L^2(B_8)}.
	\end{equation}
	Given $0 < t <1$ we define the function
	\begin{equation*}
		(\delta_2^{\pm}w_r)(x) = \frac{w_r(x\pm te_2)-w_r(x)}{t}
	\end{equation*}
	and we observe that 
	\begin{equation}
		\label{equation:estimate:quotient}
		\norm{\delta_2^{\pm}w_r}_{L^2(B_{2})} \leq C \norm{w_r}_{L^2(B_8)}
	\end{equation}
	as a consequence of \eqref{equation:caccioppolli}. Now since $p(x) = \frac{1}{2}x_1^2$ we have that $\delta_2^{\pm} w_r = \delta_2^{\pm} u_r$, and so we find that $\Delta(\delta_2^{\pm} w_r) \leq 0$ in $\{u_r > 0\} \cap B_4$ and $\delta_2^{\pm}w_r \geq 0$ in $\{u_r = 0\} \cap B_4$. This implies that $\min\{\delta_2^{\pm}w_r,0\}$ is superharmonic and we obtain using \eqref{equation:estimate:quotient} that
	\begin{equation*}
		\min_{B_{2}} \delta_2^{\pm}w_r \geq - C \norm{w_r}_{L^2(B_8)}.
	\end{equation*}
	Now since $\delta_2^{\pm}w_r(x \mp te_j) = -(\delta_2^{\mp}w_r)(x)$ we find
	\begin{equation*}
		\max_{B_{1}} \delta_2^{\pm}w_r \leq C \norm{w_r}_{L^2(B_8)}
	\end{equation*}
	which concludes Step 3. Now Lemma \ref{lemma:holder:characterisation} can be applied and we obtain, using \eqref{equation:Linfty:L2:bound}, the desired estimate.
\end{proof}

\section{Proof of Proposition \ref{expansionthm}}\label{section:appendix:expansion}
In this appendix we prove Proposition \ref{expansionthm}. However, we note that using the estimates on the coincidence set obtained in Proposition \ref{Cdelta}, the following computations (and hence the expansion) can be established for any $x_2$-monotone global solution as in Definition \ref{x2mondef}.
Recall that
\begin{equation*}
	G(x,y)=\log(\abs{x-y})-\log(\abs{y})+\frac{x\cdot y}{\abs{y}^2}
\end{equation*}
and note that 
\begin{equation}\label{fundamentalsolution}
	-\frac{1}{2\pi}\Delta_x G(x,y)=-\delta_x
\end{equation}
in the sense of distributions while
\begin{equation}
	\label{potential:zero}
	G(0,y) = \abs{\nabla_x G(0,y)} = 0.
\end{equation}

The main step in the proof of Proposition \ref{expansionthm} consists in showing that $V_{\gamma P}(x)$ is subquadratic. Once this is achieved the proof can be carried out using a Liouville argument exactly as in the proof of \cite[Proposition 3.4]{eberle2022complete}. 

\begin{lemma}\label{subquadraticbehaviour} 
	There exists a constant $C=C(\gamma)$ such that
	\begin{equation*}
		\abs{V_{\gamma P}(x)}\leq C(1+\abs{x}^{\frac74}) \hspace{2mm} \text{for all $x\in \R^2$}.
	\end{equation*}
\end{lemma}

\begin{proof}
	Throughout this proof we will use $C$ to denote a constant that depends on $\delta$ and which may change from line to line. The first observation to make is that, by a Taylor expansion with remainder, we have
	\begin{equation}\label{taylorbound}
		\abs{G(x,y)}\leq C\abs{x}^2\int_{0}^1 \frac{1}{\abs{y - \tau x}^2} d\tau.
	\end{equation}
	Indeed, setting $f(\tau) = \ln(\abs{\tau x -y})$ we have
	\begin{align*}
		f'(\tau) &= \frac{1}{\abs{\tau x- y}^2} (\tau x -y)\cdot x \\
		f''(\tau) &= \frac{\abs{x}^2}{\abs{\tau x- y}^2} - 2 \frac{1}{\abs{\tau x- y}^4} \left[(\tau x -y)\cdot x\right]^2,
	\end{align*}
	from which \eqref{taylorbound} easily follows. In particular if $\abs{y}\geq 2\abs{x}$ we have that $\abs{y - \tau x} \geq \frac12\abs{y}$ and we have 
	\begin{equation} \label{taylorbound2}
		\abs{G(x,y)} \leq C\frac{\abs{x}^2}{\abs{y}^2}.
	\end{equation}
	Fixing now $\hat{r} \geq 1$ such that $\log(r)\leq r^{\frac14}$ for $r\geq \hat{r}$, we observe that for each $x \in \R^2$ we can split $V_{\gamma P}(x)$ up as
	\begin{align*}
		V_{\gamma P}(x) =  \int_{\gamma P \cap B_{2\hat{r}}} G(x,y) dy + \int_{\gamma P\cap (B_{2\abs{x}}\backslash B_{2\hat{r}})} G(x,y) dy+  \int_{\gamma P\backslash B_{2\abs{x}}} G(x,y) dy.
	\end{align*}
	In the third integral we have $\abs{y} \geq 2\abs{x}$ and so we can use \eqref{taylorbound2} to obtain
	\begin{equation}\label{thirdintegral}
		\abs{\int_{\gamma P\backslash B_{2\abs{x}}} G(x,y) dy} \leq C\abs{x}^2 \int_{2\abs{x}}^{\infty} \int_{-(\gamma y_2)^{\frac{1}{2}}}^{(\gamma y_2)^{\frac{1}{2}}} \frac{1}{y_2^2} dy_1 dy_2 \leq C\gamma^{\frac12}\abs{x}^{\frac{3}{2}},
	\end{equation}
	For the first integral we first estimate from above,
	\begin{align*}
		\int_{\gamma P\cap B_{2\hat{r}}} G(x,y)
		&=\int_{\gamma P\cap B_{2\hat{r}}} \log(\abs{x-y}) - \int_{\gamma P\cap B_{2\hat{r}}} \log(\abs{y}) + \int_{\gamma P\cap B_{2\hat{r}}} \frac{x\cdot y}{\abs{y}^2} \\
		&\leq \int_{(\gamma P\cap B_{2\hat{r}})\backslash B_1(x)} \log(\abs{x-y}) - \int_{(\gamma P\cap B_{2\hat{r}})\cap B_1} \log(\abs{y}) + \int_{B_{2\hat{r}}} \frac{\abs{x}}{\abs{y}} \\
		&\leq C(1+\abs{x}),
	\end{align*}
	where we used that in $(\gamma P\cap B_{2\hat{r}})\backslash B_1(x)$ we have that $1 \leq \abs{x-y}\leq C\hat{r} + \abs{x}$. Similarly we obtain the lower bound 
	\begin{align*}
		\int_{\gamma P\cap B_{2\hat{r}}} G(x,y)
		&\geq - \int_{B_{2\hat{r}}} \frac{\abs{x}}{\abs{y}} + \int_{B_1(x)}\log(\abs{x-y}) - \int_{B_{2\hat{r}}\backslash B_1}  \log(\abs{y}) \\
		&\geq -C(1+\abs{x}),
	\end{align*}
	and we conclude that 
	\begin{equation}\label{firstintegral}
		\abs{\int_{C \cap B_{2\hat{r}}} G(x,y) dy } \leq C(1+\abs{x}).
	\end{equation}
	Finally the middle integral is non-zero only if $\abs{x} \geq \hat{r} > 0$ and so we can write $x = \abs{x} z$ for some $z \in \p B_1$ and we find using the scaling properties of $G(x,y)$ that
	\begin{align*}
		\abs{ \int_{\gamma P\cap (B_{2\abs{x}}\backslash B_{2\hat{r}})} G(x,y) dy}
		&\leq \int_{\{\abs{y_1} \leq (\gamma y_2)^{\frac{1}{2}}\} \cap \{\hat{r} \leq y_2 \leq 2\abs{x}\}} \abs{G(\abs{x}z,y)} dy\\
		&= \int_{\{\abs{y_1} \leq (\gamma y_2)^{\frac{1}{2}}\} \cap \{\hat{r} \leq y_2 \leq 2\abs{x}\}} \abs{G(z,\frac{y}{\abs{x}})} dy\\
		&\leq \abs{x}^2 \int_{\{ \abs{y_1} \leq \abs{x}^{-\frac{1}{2}} (\gamma y_2)^{\frac{1}{2}} \} \cap \{\frac{\hat{r}}{\abs{x}} \leq y_2 \leq 2\}} \abs{\log(\abs{z-y})} + \abs{\log(\abs{y})} + \frac{\abs{z}}{\abs{y}} dy.
	\end{align*}
	Since we are integrating now over a bounded domain (and $\frac{1}{\abs{y}}$ and $\ln(\abs{y})$ are integrable around zero) we can separate this integral. Dealing with the second and third terms first, we find
	\begin{align*}
		\abs{x}^2 \int_{\{ \abs{y_1} \leq \abs{x}^{-\frac{1}{2}} (\gamma y_2)^{\frac{1}{2}} \} \cap \{\frac{\hat{r}}{\abs{x}} \leq y_2 \leq 2\}}  \abs{\log(\abs{y})} + \frac{\abs{z}}{\abs{y}} dy
		& \leq \abs{x}^2 \int_{\frac{\hat{r}}{\abs{x}}}^{2} \int_{-\abs{x}^{-\frac{1}{2}} (\gamma y_2)^{\frac{1}{2}}}^{\abs{x}^{-\frac{1}{2}} (\gamma y_2)^{\frac{1}{2}}} \log(\hat{r}\abs{x}) + \frac{1}{y_2} dy \\
		& \leq C \abs{x}^{\frac{3}{2}}\log(\abs{x}).
	\end{align*}
	Now for the first component of the integral we observe
	\begin{align*}
		\abs{x}^2 \int_{\{ \abs{y_1} \leq \abs{x}^{-\frac{1}{2}} (\gamma y_2)^{\frac{1}{2}} \} \cap \{\frac{\hat{r}}{\abs{x}} \leq y_2 \leq 2\}} &\abs{\log(\abs{z-y})} dy \\
		& =  \int_{\{ \abs{y_1} \leq (\gamma y_2)^{\frac{1}{2}} \} \cap \{\hat{r} \leq y_2 \leq 2\abs{x}\}} \abs{\log(\frac{1}{\abs{x}}) + \log(\abs{x-y})}dy \\
		&\leq \int_{ \{\abs{y_1} \leq (\gamma y_2)^{\frac{1}{2}} \} \cap \{\hat{r} \leq y_2 \leq 2\abs{x}\}} \log(\abs{x}) -\int_{B_1(x)} \log(\abs{x-y})dy \\&+ \int_{( \{\abs{y_1} \leq (\gamma y_2)^{\frac{1}{2}} \} \cap \{\hat{r} \leq y_2 \leq 2\abs{x}\}) \backslash B_1(x)} \log(\abs{x-y})dy  \\
		& \leq C\abs{x}^{\frac{3}{2}}\log(\abs{x}),
	\end{align*}
	where in the last step we used that $1 \leq \abs{x -y} \leq C\abs{x}$ in the region $(\{  \abs{y_1} \leq (\gamma y_2)^{\frac{1}{2}} \}\cap \{\hat{r} \leq y_2 \leq 2\abs{x}\}) \backslash B_1(x)$. Since $\log(\abs{x}) \leq \abs{x}^{\frac14}$ for $\abs{x} \geq \hat{r}$, we obtain the estimate
	\begin{equation}
		\label{middleintegral}
		\abs{ \int_{C\cap (B_{2\abs{x}}\backslash B_{2\hat{r}})} G(x,y) dy} \leq  C\abs{x}^{\frac{7}{4}}.
	\end{equation}
	Putting \eqref{thirdintegral}, \eqref{firstintegral} and \eqref{middleintegral} together we obtain the result.
\end{proof}

\begin{proof}[Proof of Proposition \ref{expansionthm}]
	For each $\rho > 0$ we have that $\Delta V_{\gamma P \cap B_{\rho}}(x) = -\chi_{\gamma P \cap B_{\rho}} \in L^{\infty}(\R^n)$ and so by \cite[Proposition 2.18]{fernandez2020regularity} applied in each ball $B_R$, we have that $V_{\gamma P \cap B_{\rho}}\to V_{\gamma P}$ and $\nabla V_{\gamma P \cap B_{\rho}} \to \nabla V_{\gamma P}$ locally uniformly in $\R^n$ as $\rho \to \infty$. By \eqref{potential:zero} we therefore have that 
	\begin{equation}\label{eqn:potential:zeros}
		V_{\gamma P}(0)  = 0 \hspace{2mm}\text{and}\hspace{2mm} \abs{\nabla V_{\gamma P} (0) } =0.
	\end{equation}
	Now, since $\Delta (u_{\gamma} - p) = - \chi_{\gamma P}$ we have that $\Delta (u-p - V_{\gamma P}) =0$ and since both $u-p$ and $V_{\gamma P}$ grow subquadratically as $\abs{x} \to \infty$, we have by Liouville's theorem that
	\begin{equation*}
		u = p + V_{\gamma P}(x) + \ell(x) +c
	\end{equation*}
	where $\ell: \R^n \to \R$ is a linear function and $c \in \R$. By \eqref{eqn:potential:zeros}, $\ell = c = 0$ and this concludes the proof.
\end{proof}
\end{appendices}

\bibliographystyle{alpha}
\bibliography{References.bib}

\begin{bibdiv}
\begin{biblist}

\bib{acf84}{article}{
      author={Alt, Hans~Wilhelm},
      author={Caffarelli, Luis~A.},
      author={Friedman, Avner},
       title={Variational problems with two phases and their free boundaries},
        date={1984},
        ISSN={0002-9947,1088-6850},
     journal={Trans. Amer. Math. Soc.},
      volume={282},
      number={2},
       pages={431\ndash 461},
         url={https://doi.org/10.2307/1999245},
      review={\MR{732100}},
}

\bib{caffarelli1998obstacle}{article}{
      author={Caffarelli, Luis~A.},
       title={The obstacle problem revisited},
        date={1998},
     journal={Journal of Fourier Analysis and Applications},
      volume={4},
      number={4},
       pages={383\ndash 402},
}

\bib{caffarelli2000regularity}{article}{
      author={Caffarelli, Luis~A},
      author={Karp, Lavi},
      author={Shahgholian, Henrik},
       title={Regularity of a free boundary with application to the pompeiu
  problem},
        date={2000},
     journal={Annals of Mathematics},
       pages={269\ndash 292},
}

\bib{caffarelli2005geometric}{book}{
      author={Caffarelli, Luis~A},
      author={Salsa, Sandro},
      author={Salsa, S},
       title={A geometric approach to free boundary problems},
   publisher={American Mathematical Soc.},
        date={2005},
      volume={68},
}

\bib{di1986bubble}{article}{
      author={Di~Benedetto, Emmanuele},
      author={Friedman, Avner},
       title={Bubble growth in porous media},
        date={1986},
     journal={Indiana University mathematics journal},
      volume={35},
      number={3},
       pages={573\ndash 606},
}

\bib{dive1931attraction}{article}{
      author={Dive, Pierre},
       title={Attraction des ellipsoides homog{\`e}nes et r{\'e}ciproques d'un
  th{\'e}or{\`e}me de newton},
        date={1931},
     journal={Bulletin de la Soci{\'e}t{\'e} Math{\'e}matique de France},
      volume={59},
       pages={128\ndash 140},
}

\bib{eberle2022complete}{article}{
      author={Eberle, Simon},
      author={Figalli, Alessio},
      author={Weiss, Georg~S},
       title={Complete classification of global solutions to the obstacle
  problem},
        date={2022},
     journal={arXiv preprint arXiv:2208.03108},
}

\bib{eberleduke2022}{article}{
      author={Eberle, Simon},
      author={Shahgholian, Henrik},
      author={Weiss, Georg~S},
       title={On global solutions of the obstacle problem},
        date={2023},
     journal={Duke Mathematical Journal},
      volume={1},
      number={1},
       pages={1\ndash 45},
}

\bib{eberle2021characterizing}{article}{
      author={Eberle, Simon},
      author={Weiss, G},
       title={Characterizing compact coincidence sets in the obstacle
  problem—a short proof},
        date={2021},
     journal={St. Petersburg Mathematical Journal},
      volume={32},
      number={4},
       pages={705\ndash 711},
}

\bib{evanspde}{book}{
      author={Evans, Lawrence~C},
       title={Partial differential equations},
   publisher={American Mathematical Society},
        date={2010},
      volume={19},
}

\bib{fernandez2020regularity}{article}{
      author={Fern{\'a}ndez-Real, Xavier},
      author={Ros-Oton, Xavier},
       title={Regularity theory for elliptic pde},
        date={2020},
     journal={Forthcoming book},
}

\bib{figalli2020generic}{article}{
      author={Figalli, Alessio},
      author={Ros-Oton, Xavier},
      author={Serra, Joaquim},
       title={Generic regularity of free boundaries for the obstacle problem},
        date={2020},
     journal={Publications math{\'e}matiques de l'IH{\'E}S},
      volume={132},
      number={1},
       pages={181\ndash 292},
}

\bib{figalli2019fine}{article}{
      author={Figalli, Alessio},
      author={Serra, Joaquim},
       title={On the fine structure of the free boundary for the classical
  obstacle problem},
        date={2019},
     journal={Inventiones mathematicae},
      volume={215},
      number={1},
       pages={311\ndash 366},
}

\bib{franceschini2021c}{article}{
      author={Franceschini, Federico},
      author={Zato{\'n}, Wiktoria},
       title={${C}^{\infty} $ partial regularity of the singular set in the
  obstacle problem},
        date={to appear},
     journal={Anal. PDE},
}

\bib{friedman1986characterization}{article}{
      author={Friedman, Avner},
      author={Sakai, Makoto},
       title={A characterization of null quadrature domains in r n},
        date={1986},
     journal={Indiana University mathematics journal},
      volume={35},
      number={3},
       pages={607\ndash 610},
}

\bib{lewy1979inversion}{article}{
      author={Lewy, Hans},
       title={An inversion of the obstacle problem and its explicit solution},
        date={1979},
     journal={Annali della Scuola Normale Superiore di Pisa-Classe di Scienze},
      volume={6},
      number={4},
       pages={561\ndash 571},
}

\bib{petrosyan2012regularity}{book}{
      author={Petrosyan, Arshak},
      author={Shahgholian, Henrik},
      author={Uraltseva, Nina~Nikolaevna},
       title={Regularity of free boundaries in obstacle-type problems},
      series={Graduate Studies in Mathematics},
   publisher={American Mathematical Soc.},
        date={2012},
      volume={136},
}

\bib{sakai1981}{article}{
      author={Sakai, M.},
       title={Null quadrature domains},
        date={1981},
     journal={J. Anal. Math.},
      volume={40},
       pages={144–154},
}

\bib{shahgholian1992quadrature}{article}{
      author={Shahgholian, Henrik},
       title={On quadrature domains and the schwarz potential},
        date={1992},
     journal={Journal of mathematical analysis and applications},
      volume={171},
      number={1},
       pages={61\ndash 78},
}

\end{biblist}
\end{bibdiv}
\end{document}